\documentclass[journal]{IEEEtran}

\usepackage{cite}
\usepackage[pdftex]{graphicx}
\usepackage[cmex10]{amsmath}
\usepackage{algorithmicx}
\usepackage{algpseudocode}
\usepackage{array}
\usepackage{fixltx2e}
\usepackage{url}
\usepackage{multirow}
\usepackage{bigstrut}
\usepackage{colortbl}
\usepackage{amsmath}
\usepackage{amsthm}
\usepackage{amssymb}
\usepackage{amsfonts}
\usepackage{ulem}

\newenvironment{nproof}{{\noindent \it Proof:}}{\hfill $\square$}

\newtheorem{problem}{Problem}

\newtheorem{theorem}{Theorem}
\newtheorem{corollary}{Corollary}
\newtheorem{property}{Property}
\newtheorem{proposition}{Proposition}
\newtheorem*{propositionA-1}{Proposition A-1}
\newtheorem*{propositionA-2}{Proposition A-2}
\begin{document}
%
\title{A Study on the Block Relocation Problem: Lower Bound Derivations and Strong Formulations}
%
%
%

\author{Chao~Lu,
Bo~Zeng,~\IEEEmembership{Member,~IEEE,}
and~Shixin~Liu
\thanks{This work was supported by the China Scholarship Council scholarship,
supported in part by National Key R\&D Program of China under Grant No. 2017YFB0306400,
National Natural Science Foundation of China under Grant No. 61573089, 61703220, 71871105,
Shandong Provincial Natural Science Foundation,
China under Grant No ZR2016FP02,
Postdoctoral Science Foundation Project of China under Grant No 2017M610407
and Qingdao Postdoctoral Research Project under Grant No 2016027.}
\thanks{C. Lu and S. Liu are with the State Key Laboratory of Synthetical Automation for Process Industries,
as well as the College of Information Science and Engineering,
Northeastern University, Shenyang 110819, China\ \ (e-mail: surpassu@live.com, sxliu@mail.neu.edu.cn).}
\thanks{B. Zeng is with the Department of Industrial Engineering and the Department of Electrical and Computer Engineering,
Swanson School of Engineering, University of Pittsburgh, Pittsburgh, PA 15261, USA (email: bzeng@pitt.edu).}
%
}

\maketitle


\begin{abstract}

The \textit{block relocation problem} (BRP) is a fundamental operational issue in modern warehouse and yard management, which, however, is very challenging to solve.  In this paper, to advance our understanding on this problem and to provide a substantial assistance to practice, we $(i)$ introduce a classification scheme and present a rather comprehensive review on all 16 BRP variants; $(ii)$ develop a general framework to derive lower bounds on the number of necessary relocations and demonstrate its connection to existing ones on the unrestricted BRP variants; $(iii)$ propose and employ a couple of new critical substructures concepts to analyze the BRP and obtain a lower bound that dominates  all existing ones; $(iv)$ build a new and strong mixed integer programming (MIP) formulation that is adaptable to compute 8 BRP variants, and design a novel MIP formulation based iterative procedure to compute exact BRP solutions; $(v)$ extend the MIP formulation to address four typical industrial considerations. Computational results on standard test instances show that the new lower bound is significantly stronger, and our new MIP computational methods have superior performances over a state-of-the-art formulation.
\end{abstract}

\IEEEpeerreviewmaketitle

\section{Introduction}
%
%
%
%

\IEEEPARstart{T}{he} block relocation problem (BRP) is a fundamental operational issue in modern warehouse and yard management, especially for material handling in a container yard or a steel slab yard.
For example, in a container yard, heavy and large containers, i.e., blocks in this context, are stored temporarily in stacks (i.e., columns) as in Figure \ref{figure: container_yard}.
Before those containers can be shipped to different destinations, they will be retrieved according to the \textit{prioritized retrieval list}.
One prioritized list is illustrated by the numbers on containers in Figure \ref{figure: container_yard}, where the smaller number the higher priority.
Clearly, if a container of a lower priority is piled on top of another one with a higher priority, e.g., container 13 is on top of container 12, retrieving the latter one can only be done after moving the former one to somewhere else (typically to another stack).
Moving a container (or block in generally) from a stack to another one is often referred to as a \textit{relocation}~\cite{Kim-2006-COR}.
In practice, as containers, steel slabs and other blocks are large and heavy, moving them needs powerful and expensive handling equipment, and the associated operations are time and energy consuming.
Hence, to retrieve blocks from the yard following their retrieval priorities, an essential issue is to determine a move sequence to complete the task with the least number of relocations,  which is referred to as the aforementioned BRP \cite{Caserta-2011-ORS} or \textit{container relocation problem} \cite{Forster-2012-COR} if specified to containers.

With the rapid automation of warehouse and yard operations, the  BRP and its different variants have received a lot of attention from engineers and scholars, and many studies have been published in the literature after its formal introduction~\cite{Kim-2006-COR} in 2006.
For example, many well-defined mixed integer programming (MIP) formulations and sophisticated exact or heuristic algorithms have been designed and analyzed (e.g., \cite{Carlo-2014-EJOR, Lehnfeld-2014-EJOR} and references therein).
Nevertheless, the BRP has been proven to be NP-hard and is computationally very challenging for practical-scale instances  \cite{Caserta-2012-EJOR}. According to our numerical study, a state-of-the-art formulation may take hours to derive a feasible relocation plan for a rather small-scale instance. Certainly, such a computational performance does not ensure its application in practice.
With little quantitative support,  the current practice is often based on operators' experience or following fixed relocation rules, leading to many unnecessary relocations and a heavy operational burden.

To change such a situation, especially to provide a substantial assistance to practice, we address in this paper two critical issues of the BRP, i.e.,
a stronger lower bound on the number of necessary relocations,
and a computationally more effective mathematical formulation.
Indeed, we develop a general framework to  understand the number of necessary relocations, which interprets all known lower bounds on that number and leads us to derive a much stronger lower bound.
{\color{black} Moreover, a deep insight from our new formulation inspires us to develop a novel iterative computational procedure.
Also, its flexibility and strength are demonstrated by  extending it to address four additional industrial considerations.}
Overall, we mention that our new results either theoretically dominate the state-of-the-art in the literature or drastically outperform existing formulations.

\begin{figure}[!t]
\begin{center}
\includegraphics[width=0.89\linewidth]{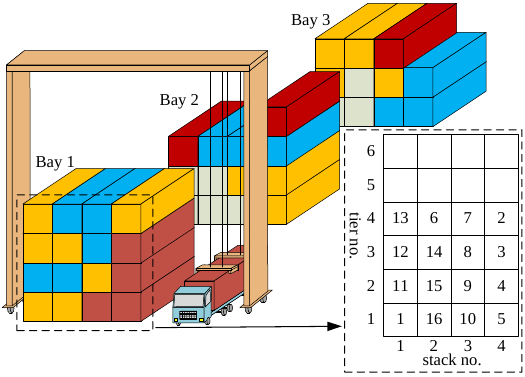}
\caption{Containers piled in stacks with their retrieval priorities}
\label{figure: container_yard}
\end{center}
\end{figure}

The remaining part of this paper is organized as follows.
In Section \ref{section_literature_review}, we classify variants of the BRP into a few classifications and review  existing literature.
In Section \ref{section_LB}, we  present a general framework to derive lower bounds on the number of necessary relocations, demonstrate its connection to existing ones, and apply it to analyze critical structures and obtain a stronger lower bound.
In Section \ref{section_formulation}, we build an MIP formulation for the BRP.
Also, a novel MIP formulation based iterative procedure is developed to compute exact BRP solutions.
{\color{black}
We extend the MIP formulation to address additional industrial considerations in Section \ref{section_customizations}.
}
Performances of our lower bound and computational methods are reported in Section~\ref{section_experiments}.
Section \ref{section_conclusion} concludes this paper with a discussion on future research.

\section{Literature Review}
\label{section_literature_review}
Up to now, many solid studies have been published on different variants of the BRP.
In this section, we classify existing publications according to their {\color{black}nature}, review those publications, and describe their significant contributions.

\subsection{Problem Classification}

According to physical restrictions,
retrieval specifications, and the capacity of handling equipment,
there exist four major features in defining and formulating a BRP model.
They are $(\romannumeral1)$ what restrictions are imposed on moves,
$(\romannumeral2)$ whether retrieval priorities of blocks are distinct,
$(\romannumeral3)$ whether all the blocks are to be retrieved,
and $(\romannumeral4)$ whether only one block can be moved at a time.
We explain them  in the following.

$(\romannumeral1)$ \textit{Restricted vs. Unrestricted}:
In practice, the first block in the \textit{current} prioritized retrieval list  is often referred to as the \textit{target block}.
As mentioned, the target block can only be retrieved after blocks piled above it are relocated to other stacks.
Those necessary relocations are called \textit{forced moves}~\cite{Lehnfeld-2014-EJOR}.
If only forced moves are allowed to retrieve every block in the list, the BRP is {\color{black}called} the restricted one. Otherwise, it is called the unrestricted one.

$(\romannumeral2)$ \textit{Distinct vs. Duplicate}:
If each block is of a distinct priority, the BRP is called the BRP with distinct priorities.
Otherwise, it is called the BRP with duplicate priorities.

$(\romannumeral3)$ \textit{Complete vs. Incomplete}:
If all blocks in one row (i.e., one bay) are to be retrieved, the problem is called the BRP with the complete retrieval.
Otherwise, it is called the BRP with the incomplete retrieval.

$(\romannumeral4)$ \textit{Individual vs. Batch}:
If the handling equipment (e.g., a crane or straddle
carrier) can move exactly one block at a time, the BRP is called the BRP with individual moves.  Otherwise, it is called the BRP with batch moves \cite{Zhang-2016-FSMJ}. The latter one is often seen in steel industry.

\begin{table}[!t]
\centering
\caption{Classification of Existing Literature on the BRP by the Newly Proposed Four-Characteristic Scheme}
\label{table: classification-of-Procedure}
\renewcommand\tabcolsep{3pt}
\begin{tabular}{ccp{1.1in}<{\centering}p{1.4in}<{\centering}}
\hline
\multicolumn{1}{c}{var} & \multicolumn{1}{c}{characteristics} & \multicolumn{1}{c}{literature} & \multicolumn{1}{c}{years (20yy)} \bigstrut\\
\hline
\rowcolor[rgb]{ .867,  .867,  .867}1    & res\,$|$dis\,\,$|$com$|$ind
& \cite{Kim-2006-COR, Caserta-2009-Conference, Wan-2009-NRL, Wu-2010-Conference, Zhang-2010-Conference, Caserta-2012-EJOR, Zhu-2012-IEEET-ASE, Jovanovic-2014-CIE, Zehendner-2014-IJPR, Borjian-2015-arXiv, Zehendner-2015-EJOR, Exposito-2015-ESA, Tang-2015-IIET, Ku-2016-COR, Tanaka-2016-IEEET-ASE, Galle-2016-ORL, Quispe-2018-COR, Silva-2018-EJOR, Galle-2018-EJOR}
& 06, 09, 10, 12, 14, 15, 16, 18\bigstrut[t]\\
2    & \ \ \ \ \,dis\,\,$|$com$|$bat & not found\\
3    & \ \ \ \ \,dis\,\,$|$inc\,\,\,$|$ind
& \cite{Caserta-2011-ORS, Exposito-2014-AEI}
& 11, 14\\
4    & \ \ \ \ \,dis\,\,$|$inc\,\,\,$|$bat & not found\\
5    & \ \ \ \ \,dup$|$com$|$ind
& \cite{Kim-2006-COR, Tanaka-2016-IEEET-ASE, Silva-2018-EJOR}
& 06, 16, 18\\
6    & \ \ \ \ \,dup$|$com$|$bat & not found\\
7    & \ \ \ \ \,dup$|$inc\,\,\,$|$ind & not found\\
8    & \ \ \ \ \,dup$|$inc\,\,\,$|$bat & not found\\
\rowcolor[rgb]{ .867,  .867,  .867} 9    & unr$|$dis\,\,$|$com$|$ind
& \cite{Caserta-2012-EJOR, Zhu-2012-IEEET-ASE, Petering-2013-EJOR, Tanaka-2015-Conference, Tanaka-2015-Conference2, Tanaka-2018-COR, Tricoire-2018-COR, Silva-2018-EJOR, Feillet-2019-COR}
& 12, 13, 15, 18, {\color{black}19}\\
10   & \ \ \ \ \,dis\,\,$|$com$|$bat & \cite{Zhang-2016-FSMJ}
& 16\\
11   & \ \ \ \ \,dis\,\,$|$inc\,\,\,$|$ind &  \cite{Exposito-2014-AEI}
& 14\\
12   & \ \ \ \ \,dis\,\,$|$inc\,\,\,$|$bat & not found\\
13   & \ \ \ \ \,dup$|$com$|$ind & \cite{Forster-2012-COR, Jin-2013-Conference, Jin-2015-EJOR, Silva-2018-EJOR}
& 12, 13, 15, 18\\
14   & \ \ \ \ \,dup$|$com$|$bat & not found\\
15   & \ \ \ \ \,dup$|$inc\,\,\,$|$ind & not found\\
16   & \ \ \ \ \,dup$|$inc\,\,\,$|$bat & not found \bigstrut[b]\\
\hline
\end{tabular}\\
\ \\ \leftline{var = variation; res = restricted, unr = unrestricted;}
\leftline{dis = distinct priorities, dup = duplicate priorities;}
\leftline{com = complete retrieval, inc = incomplete retrieval;}
\leftline{ind = individual moves, bat = batch moves.}
\end{table}%

Clearly,  there are $2^4=16$ different variants based on particular specifications on those four features.
Accordingly, we group subjects of existing publications as in Table \ref{table: classification-of-Procedure} after performing a rather exhaustive review.
We mention that studies on the BRP problems with other extensions, e.g., those with stochastic factors or vehicle routing decisions, are not included.
Among those in Table \ref{table: classification-of-Procedure}, variants 1 and 9 are most popular, i.e., the restricted BRP and the unrestricted BRP with distinct priorities, the complete retrieval and individual moves.
The reason behind {\color{black}it} is that they have the fundamental structures that do not depend on particular working conditions or facilities.
Moreover, we can argue that other variants are relaxations of them.
For example, we can convert an instance of variant 15, i.e., the unrestricted BRP with duplicate priorities, the incomplete retrieval and individual moves, to an instance of variant 9 by assigning distinct priorities to blocks of the same priority and considering no-to-retrieve blocks with the lowest priorities.
\textcolor{black} {Certainly, solving the latter instance does not necessarily ensure an optimal solution of the former one.}
However, any feasible solution to the latter instance is always feasible to the former one, if retrieval moves of no-to-retrieve blocks are ignored.
Hence, computing variant 1 or 9 provides a basic strategy to handle more involved variants.

Although many research efforts have been devoted to the BRP variants, as noted in the following reviews, existing results might not be able to efficiently deal with their practical instances, which, therefore, inspires us to perform a study to gain a deeper understanding and to develop efficient solution methods.

\subsection{Literature on Restricted BRP Variants}

We first give a review on existing studies on variant 1, which is the default BRP in this subsection, and then describe relevant work on variants 3 and 5.

1) \textit{Theoretical analysis}.
To the best of our knowledge, the study in \cite{Kim-2006-COR} is the first analytical one in the literature.
As the number of relocations is the primary concern of the BRP, they give a lower bound on this number.
Since then, this lower bound has been successively improved by different scholars \cite{Zhang-2010-Conference, Zhu-2012-IEEET-ASE, Tanaka-2016-IEEET-ASE, Quispe-2018-COR}.
Although the lower bound of \cite{Kim-2006-COR} is rather weak, Galle et al.~\cite{Galle-2016-ORL} show that the expected minimum number of relocations approaches to it if the number of stacks grows to infinite and the priorities of blocks are uniformly distributed.
Additionally, some upper bound estimations on that number have also been developed \cite{Caserta-2012-EJOR, Zehendner-2014-IJPR, Zehendner-2015-EJOR}.

2) \textit{Exact tree search algorithms}.
To directly solve the BRP problems, we note that there are three main types of tree search based exact algorithms,
\textcolor{black} {which are the fastest exact algorithms up to now.}
They are branch-and-bound (B\&B) algorithms, A* based algorithms, and other tree search algorithms.
Simple B\&B algorithms are developed by Kim and Hong \cite{Kim-2006-COR} and Wu and Ting \cite{Wu-2010-Conference}.
More sophisticated B\&B algorithms are developed by Exp{\'o}sito-Izquierdo et al. \cite{Exposito-2015-ESA} and Tanaka and Takii \cite{Tanaka-2016-IEEET-ASE}.
Along with B\&B algorithms, Zhang et al. \cite{Zhang-2010-Conference} propose iterative deepening A* algorithms (IDA) that take advantages of two new lower bounds and several probe heuristics \cite{Zhu-2012-IEEET-ASE}.
Since then, a couple of more A* algorithms have been introduced, including \cite{Borjian-2015-arXiv} where an A* algorithm makes use of existing lower bounds and an existing upper bound, and \cite{Quispe-2018-COR} where new lower bounds and several existing lower bounds are integrated for a better performance.
Finally, we note that Ku and Arthanari \cite{Ku-2016-COR} design a bidirectional search algorithm, which incorporates a search space reduction technique, called the abstraction method, within a tree search algorithm.

3) \textit{Mathematical programming formulations}.
\textcolor{black} {In the literature, the BRP is often formulated into MIPs that can be computed by state-of-the-art solvers or packages. Note that those solvers or packages, different from the aforementioned particular algorithms, compute general mathematical programs, which allow users to flexibly augment basic MIP models with additional considerations and concerns arising from practice.}%
To the best of our knowledge, Wan et al. \cite{Wan-2009-NRL} develop the first binary formulation, which is then improved by Tang et al. \cite{Tang-2015-IIET} with a significantly better computational performance.
Caserta et al. \cite{Caserta-2012-EJOR} present another binary formulation called BRP-\uppercase\expandafter{\romannumeral2}.
Later, it is improved in~\cite{Exposito-2015-ESA} by replacing some constraints, and is enhanced in \cite{Zehendner-2015-EJOR} by removing superfluous variables, tightening some constraints, and applying a pre-processing step to fix several variables.
We mention that a relocation sequence based reformulation is proposed in Zehendner and Feillet \cite{Zehendner-2014-IJPR} to support a column generation algorithm for the BRP.
Additionally, a couple of stronger binary formulations are proposed very recently by Galle et al. \cite{Galle-2018-EJOR} and da Silva et al. \cite{Silva-2018-EJOR}.

4) \textit{Heuristic solution procedures}.
Because of the complexity of the BRP, most of the existing heuristics are ruled-based heuristics \cite{Kim-2006-COR, Caserta-2012-EJOR, Tang-2015-IIET} and look ahead heuristics \cite{Caserta-2009-Conference, Wu-2010-Conference}.
There are also some MIP based heuristics \cite{Wan-2009-NRL}, a beam search heuristic \cite{Wu-2010-Conference} and a fast chain heuristic \cite{Jovanovic-2014-CIE} .

5) \textit{Relevant research on variants 3 and 5}.
Research on variants 3 and 5 is rather {\color{black}scarce}.
For variant 3, Caserta et al.~\cite{Caserta-2011-ORS} develop a dynamic programming algorithm, and a heuristic method, i.e., a customized corridor method, that adopts the dynamic programming algorithm as a subroutine to achieve a stronger solution capability.
Also, Exp\'{o}sito-Izquierdo et al.~\cite{Exposito-2014-AEI} design a fast knowledge-based heuristic algorithm and two exact A* search algorithms for variant 3.
In addition to their focuses on variant 1, papers~\cite{Kim-2006-COR},  \cite{Tanaka-2016-IEEET-ASE}, and \cite{Silva-2018-EJOR} present some analysis on variant 5.

\subsection{Literature on Unrestricted BRP Variants}
Similar to our review on the restricted BRP, we first focus on existing studies on variant 9, which is the default BRP in this subsection, and then describe
relevant work on variants 10, 11 and 13.

1) \textit{Theoretical analysis}.
We mention that the lower bound on the number of necessary relocations by Kim and Hong~\cite{Kim-2006-COR}, which is originally developed for variants 1 and 5, is also applicable to variant 9, and has been considered as the basis for further improvements.
Forster and Bortfeldt \cite{Forster-2012-COR} propose a stronger lower bound for variant 13, which is also applicable to variant 9.
Recently, two new stronger lower bounds are proposed by Tanaka and Mizuno \cite{Tanaka-2018-COR} and Tricoire et al. \cite{Tricoire-2018-COR}. Regarding the upper bound, Caserta et al. \cite{Caserta-2012-EJOR} propose a closed-form upper bound.
In addition to the lower bound, Tanaka and Mizuno \cite{Tanaka-2015-Conference} propose two dominance properties associated with optimal solutions to reduce solution space. Their result is further complemented by two new dominance properties presented in  Tanaka \cite{Tanaka-2015-Conference2}.

2) \textit{Exact tree search algorithms}.
In addition to their focus on the restricted BRP, Zhu et al. \cite{Zhu-2012-IEEET-ASE} also develop IDA algorithms for the unrestricted BRP.
By using their derived dominance properties, Tanaka and Mizuno \cite{Tanaka-2015-Conference} and  \cite{Tanaka-2015-Conference2} develop some strengthened B\&B algorithms  in the search tree.
Together with a new lower bound, the B\&B algorithms are further improved in Tanaka and Mizuno \cite{Tanaka-2018-COR}. %
A recent B\&B algorithm for the BRP is developed by Tricoire et al. \cite{Tricoire-2018-COR}, which incorporates fast heuristics and another new lower bound.
\textcolor{black} {Those algorithms are again the fastest exact algorithms for this type of BRP.}

3) \textit{Mathematical programming formulations}.
Caserta et al. \cite{Caserta-2012-EJOR} develop the first binary integer program for the unrestricted BRP, which is referred to as  BRP-\uppercase\expandafter{\romannumeral1}. Note that, it could not provide a satisfactory performance even on small scale instances. Petering and Hussein \cite{Petering-2013-EJOR} present a more compact MIP formulation, which is called BRP-\uppercase\expandafter{\romannumeral3}.
Compared to BRP-\uppercase\expandafter{\romannumeral1}, BRP-\uppercase\expandafter{\romannumeral3} has much fewer integer decision variables, and demonstrates a faster computational performance. However, it can only solve 69 out of 520 benchmark instances as shown in \cite{Silva-2018-EJOR}.
Recently, da Silva et al. \cite{Silva-2018-EJOR} propose two new binary formulations, referred to as BRP-m1 and BRP-m2 respectively, both of which demonstrate significantly better computational performances over  BRP-\uppercase\expandafter{\romannumeral3}. Between them, BRP-m2 is a little bit more efficient as it can solve 181 benchmark instances while BRP-m1 can solve 154 instances.

4) \textit{Heuristic solution procedures}.
As for fast heuristics for the unrestricted BRP, Caserta et al. \cite{Caserta-2012-EJOR} propose a simple rule-based heuristic.
Petering and Hussein \cite{Petering-2013-EJOR} extend the heuristic and develop a look-ahead heuristic. %
Tricoire et al. \cite{Tricoire-2018-COR} develop four fast heuristics and a pilot method which incorporates a fast metaheuristic called rake search.
{\color{black} Feillet et al. \cite{Feillet-2019-COR} develop a local-search based heuristic, where the state space is explored by a dynamic programming algorithm.}

5) \textit{Relevant research on variants 10, 11 and 13}.
Regarding other variants, Zhang et al. \cite{Zhang-2016-FSMJ} propose a lower bound for variant 10,
and develop both  inexact and exact tree search algorithms.
Exp{\'o}sito-Izquierdo et al. \cite{Exposito-2014-AEI} develop a simple domain-specific knowledge-based heuristic for variant 11, which aims to minimize the probability that a relocated block requires new relocations in the future.
They also develop an exact A*-based search algorithm which embeds that heuristic.
For variant 13, Forster and Bortfeldt \cite{Forster-2012-COR} develop a heuristic tree search algorithm that includes a suitable branching procedure using move sequences of promising single moves. Similarly, Jin et al. \cite{Jin-2013-Conference, Jin-2015-EJOR} develop tree search based look-ahead heuristics. da Silva et al. \cite{Silva-2018-EJOR}, in addition to their focus on variant 9, also give formulations for variant 13.\\

Overall, we note in the literature that current studies on the unrestricted BRP is insufficient, and it still remains as a challenging problem.
For example, existing research on analyzing lower bounds on the number of relocations is developed rather from individual structures, with little insight to establish a systematic strategy.
Also, most benchmark instances in \cite{Silva-2018-EJOR} cannot be solved in a reasonable time using the state-of-the-art formulation.
To change such a situation of the unrestricted BRP, in this paper, we perform a study on developing a general framework to understand lower bound derivations, demonstrating its application to obtain a stronger lower bound, and constructing a computationally friendly MIP formulation, as well as an MIP formulation based exact algorithm for a faster computation.

{\color{black}
Besides the four features defining BRP variants, a practical system often has concerns or requirements due to its particular situation and environment. For example, when containers have a great variety in their weights, heavy containers should not be piled on top of light containers  \cite{Bruns-2016-EJOR}. Also, retrieval operations of steel slabs should be well paced to ensure a smooth production in the next stage \cite{Tang-2010-COR}. To illustrate its advantages in flexibility  and  general applicability in practice, our basic MIP model is modified or augmented to accommodate several practical considerations.}

\section{Derivations of Lower Bounds on the Number of Relocations}
\label{section_LB}

Different from existing studies on deriving particular lower bounds on the number of necessary relocations, we present a completely new framework to estimate that number systematically.
It reveals fundamental connections among existing lower bounds.
Then, we identify a few new results that strengthen traditional understandings.
Finally, under the proposed framework, we obtain a new lower bound that dominates all existing ones.
We believe that the overall derivation is novel, and will substantially advance our understanding on the BRP.


\subsection{A General Framework for the Derivation of Lower Bounds}


We first introduce several well-established concepts that are critical to have a deep appreciation of the BRP.

Consider one bay with $B$ blocks piled on $S$ stacks.
Let ${\color{black}\mathbb{B}}=\{1, 2, ..., B\}$ be the set of blocks, noting that a smaller ID has a higher priority, and ${\color{black}\mathbb{S}}=\{1, 2, ..., S\}$ be the set of stacks.
Also,  we denote the overall organization of those blocks, i.e., their positions in stacks, by $\mathbb{C}$. %
For a given $\mathbb{C}$, block $i$ is  called a  \textit{badly placed (BP) block} if it is piled  above some  block(s) that should be retrieved before it, i.e., $b$'s priority is lower than those of blocks below it.  %
Otherwise, $b$ is a \textit{well placed (WP) block}~\cite{Forster-2012-COR}.
Clearly, BP blocks are the causes of relocations.
For the instances displayed in Figure \ref{figure: two_instances} (a) and (b),
the blocks with priority numbers in bold and underlined are BP blocks,
and other blocks are WP blocks.

Next, we define types of block moves. Since block retrieval moves, which are mixed with relocation moves in the move sequence, are not our concern,  we only define different relocation moves. According to \cite{Forster-2012-COR}, there are four types of moves. A BB (i.e., Bad-Bad) move is a move relocating a BP block to a stack and after which the block is again a BP block.
A BG (i.e., Bad-Good) move is a move relocating a BP block to a stack and after which the block becomes a WP block. Two other moves, i.e., GB move and GG move, are defined likewise.
In Figure \ref{figure: two_instances} (a),
the move  relocating block 5 to stack 1 is a BB move,
the move relocating block 5 to stack 2 is a BG move,
the move relocating block 4 to stack 3 is a GB move,
and  the move relocating block 4 to stack 2 is a GG move.
It is straightforward that a BP block cannot be retrieved if no BG move is implemented on it.

Extending from individual blocks, we introduce the concept of \textit{the priority of stack $s$}, which is the highest priority of a block piled in stack $s$ if it is not empty, and $+\infty$ (i.e., the lowest priority) otherwise.
For the instance displayed in Figure \ref{figure: two_instances} (a),
the priorities of the four stacks are respectively 4, 13, 2 and 1 from the left to the right.
Obviously, the priority of a stack will be lower or remain the same if some block(s) is removed from it.
Another important concept is \textit{the top $k^{th}$ layer},
which consists of the top $k^{th}$ block of each stack when every stack has at least $k$ blocks.
For the instances displayed in Figure \ref{figure: two_instances} (a) and (b),
the top $2^{nd}$ layers consist of blocks $\{11, 14, 7, 1\}$ and blocks $\{6, 14, 5, 4\}$ respectively.
Similarly, we define the top $k$ layers that include all blocks from the top $1^{st}$ to the top $k^{th}$ layers.
For the instances displayed in Figure \ref{figure: two_instances} (a) and (b),
the top 2 layers include blocks $\{4, 13, 6, 5; 11, 14, 7, 1\}$ and blocks $\{16, 17, 18, 19; 6, 14, 5, 4\}$ respectively.

In the following, we present a few critical properties that are actually behind all derivations of lower bounds on the number of relocations in the BRP.
Those properties render a general framework for us to analyze lower bound derivations in a systematical fashion.
Specifically, let $f(\mathtt{B})$ be the function that returns the least number of relocations implemented on block subset $\mathtt{B}\subseteq \mathbb{B}$ across all feasible move sequences that complete the retrieval task of the given initial configuration $\mathbb{C}$.
Moreover, function $f^{\text{mt}}$ with $\text{mt}\in \{\text{BB, BG, GB, GG}\}$ returns the least number of relocations of each particular move type across all feasible move sequences.
Similarly, $f^{\overline{\text{BG}}}$ returns that least number of all non-BG moves.
\begin{theorem}
\label{thm_LB_overall}
The following inequalities hold.
\begin{align}
f(\mathbb{B}) \geq  & f^{\text{BG}}(\mathbb{B}) + f^{\overline{\text{BG}}}(\mathbb{B}) \notag\\
\geq  & f^{\text{BG}}(\mathbb{B}) + f^{\text{BB}}(\mathbb{B})+f^{\text{GB}}(\mathbb{B})+f^{\text{GG}}(\mathbb{B}) \notag
\end{align}
\end{theorem}

\begin{nproof}
Note that a relocation move must be either a BG move or a non-BG move, i.e., a BB, GB, or GG move. Nevertheless, a feasible move sequence with the least number of total relocations might have more BG moves (non-BG moves, respectively) than another feasible move sequence. Hence, according to the definitions of $f$, $f^{\text{BG}}$ and $f^{\overline{\text{BG}}}$,  the first inequality follows. By applying the same argument, we have the second inequality.
\end{nproof}

Clearly, the inequalities in Theorem \ref{thm_LB_overall} provide a useful tool in analyzing the number of relocations through
considering specific types of moves.  Indeed, this idea can be further
generalized to consider subsets of blocks. Let $\{\mathbb{B}^k_i: i = 1, \ldots, n_k\}$ be a partition
of the complete block set $\mathbb{B}$, for $k=1,2$. Then, the next result can be proven easily using the same argument presented in the proof of Theorem \ref{thm_LB_overall}.

\begin{theorem}
\label{thm_LB_overall2}
The following inequalities hold.
\begin{align}
f^{\textrm{mt}}(\mathbb{B}) \geq  & \sum_{i=1}^{n_1} {f^{\textrm{mt}}(\mathbb{B}^1_i)} \ \ \
\forall \textrm{mt}\in \{\text{BB, BG, GB, GG}\} \notag \\
f^{\overline{\text{BG}}}(\mathbb{B})\geq  & \sum_{i=1}^{n_1} f^{\overline{\text{BG}}}(\mathbb{B}^1_i) \notag\\
\geq  &\sum_{i=1}^{n_1} (f^{\textrm{BB}}(\mathbb{B}^1_i)+ f^{\textrm{GB}}(\mathbb{B}^1_i)
+f^{\textrm{GG}}(\mathbb{B}^1_i))
\notag\\
f(\mathbb{B}) \geq  & f^{\text{BG}}(\mathbb{B}) + f^{\overline{\text{BG}}}(\mathbb{B})\geq  \sum_{i=1}^{n_1} {f^{\text{BG}}(\mathbb{B}^1_i)}+\sum_{i=1}^{n_2} f^{\overline{\text{BG}}}(\mathbb{B}^2_i) \notag
\end{align}
\end{theorem}

\noindent \textbf{Remarks:} \\
(\romannumeral1) We highlight that Theorems \ref{thm_LB_overall} and \ref{thm_LB_overall2} present fundamental results.
They enable us to derive strong lower bounds through understanding and analyzing  particular move types and/or subsets that are more accessible than a complete move sequence or the whole configuration.
Indeed, basically all existing lower bounds can be obtained and interpreted easily by inequalities presented in those two theorems.
Hence, they provide a general and effective framework in performing a lower bound study.

\noindent
{\color{black}
(\romannumeral2) It is worth pointing out that this framework is generally applicable.
Without any modifications, it can be directly utilized to analyze the eight BRP variants with individual moves.
Note that for the four restricted BRP variants, $f^{\text{GB}}(\mathbb{B})$ and $f^{\text{GG}}(\mathbb{B})$ are naturally set to 0, given that there is no GB or GG move.
As a matter of fact, research on the lower bound derivation on variants with batch moves is scarce, except a study on variant 10 presented in \cite{Zhang-2016-FSMJ}.
So, one future research direction is to extend this framework with more move types to systematically study BRP variants with batch moves.}

In the following subsections,
we discuss existing lower bounds {\color{black}of the unrestricted BRP variants} and their connections to this framework,
then apply the framework to develop a new and stronger lower bound.

\subsection{A Revisit of Existing Lower Bounds}
In this subsection, we review important structural properties of the BRP that have been used in the development of four lower bounds {\color{black}of the unrestricted BRP variants} in the literature.
{\color{black}
We also discuss the computational complexities and the applicability of those lower bounds among different BRP variants, primarily among those with individual moves.
In particular, we demonstrate how those lower bounds can be derived and interpreted using the general framework in Theorems \ref{thm_LB_overall} and \ref{thm_LB_overall2}.
Without loss of generality, we assume, throughout this paper, that no directly retrievable block exists in the initial configuration $\mathbb{C}$.}

\ \\
\noindent \textit{{\color{black}1) LB$_1$: The First Fundamental Result}}

\begin{property}
\label{proty_lb1}
At least one BG move has to be implemented on a BP block.
\end{property}

\noindent$\blacksquare$ \noindent \textbf{Argument and Lower Bound Development}:\\
From the definition, it is clear that a BP block cannot be retrieved until it becomes a WP. So the property follows.

Kim and Hong \cite{Kim-2006-COR} introduce and analyze this property, and then propose a lower bound of the number of relocations.
Specifically, their lower bound is set to the number of BP blocks in the initial configuration.
The time complexity of an algorithm to compute this lower bound can be $\mathcal{O}(B)$.

{\color{black}
As the first lower bound appears in the literature, it is referred to as LB$_1$ by Tricoire et al.\cite{Tricoire-2018-COR}.
Although it is originally proposed for restricted BRP variants 1 and 5, LB$_1$ is also widely recognized as a lower bound for unrestricted BRP in the literature \cite{Forster-2012-COR} \cite{Tanaka-2018-COR} \cite{Tricoire-2018-COR}.
In fact, it can be directly applied to the four variants with the complete retrieval and individual moves.
Moreover, by simply assigning no-to-retrieve blocks with the same lowest priority, it will be able to handle the other four variants with the incomplete retrieval and individual moves.
Hence, it is applicable to all the eight variants with individual moves, including restricted and unrestricted ones. It is actually also the basis of the lower bound study \cite{Zhang-2016-FSMJ} on a variant with batch moves.}

\noindent$\blacksquare$ \noindent \textbf{Revisit and Demonstration}: \\
Let {\color{black}$\mathbb{B}^1$} be the collection of BP blocks in the initial
configuration. Property \ref{proty_lb1} can be expressed as $f^{\text{BG}}(\{b\}) \geq  1$,
for all $b\in \mathbb{B}^1$. Given the facts that $\{\{b\}: b=1,\dots, B\}$ is a partition of $\mathbb{B}$
and $\mathbb{B}^1\subset \mathbb{B}$,  we have
\begin{align*}
f(\mathbb{B}) \geq  & f^{\text{BG}}(\mathbb{B}) + f^{\overline{\text{BG}}}(\mathbb{B})
\geq   \sum_{b\in \mathbb{B}} f^{\text{BG}}(\{b\}) +
f^{\overline{\text{BG}}}(\mathbb{B}) \\
\geq  & \sum_{b\in \mathbb{B}} f^{\text{BG}}(\{b\}) \geq  \sum_{b\in \mathbb{B}^1} f^{\text{BG}}(\{b\}) \geq  |\mathbb{B}^1|,
\end{align*}
which exactly gives LB$_1$ as a valid lower bound.

\noindent$\blacksquare$ \noindent \textbf{Illustration}:\\
For the instances displayed in Figure \ref{figure: two_instances} (a) and (b),
priority numbers of BP blocks are in bold and underlined. So, {\color{black}$\mathbb{B}^1=\{5, 6, 7, 11, 12\}$}
and {\color{black}$\mathbb{B}^1=\{6, 8, 10, 12, 14, 16, 17, 18, 19\}$}, respectively,
which set LB$_1$ to {\color{black}5} and {\color{black}9}, respectively.

\begin{figure}[!t]
\begin{center}
\includegraphics{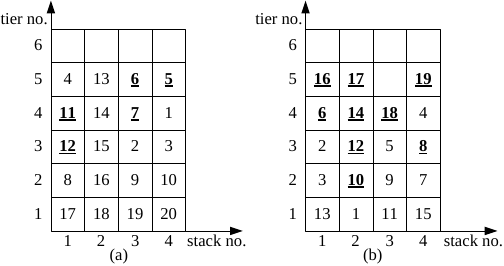}
\caption{{\color{black}Two instances with 4 stacks and a stack height limit of 6}}
\label{figure: two_instances}
\end{center}
\end{figure}

\ \\
\noindent \textit{{\color{black}2) LB$_2$: A Generalization of LB$_1$}}

\begin{property}
\label{proty_lb2}
At least one BB move has to be implemented on one block of the top $1^{st}$ layer,
if the highest priority of blocks in this layer is lower than the lowest priority of all stacks.
\end{property}

\noindent$\blacksquare$ \noindent \textbf{Argument and Lower Bound Development}:\\
If the condition of Property \ref{proty_lb2} is satisfied,
the first block to be moved is BP, and remains BP after the move,
i.e., a BB move.

Forster and Bortfeldt \cite{Forster-2012-COR} introduce and analyze this property,
and then propose a lower bound based on both Properties \ref{proty_lb1} and \ref{proty_lb2}.
Specifically, their lower bound is set to LB$_1$+1 if the condition of Property \ref{proty_lb2} is satisfied in the initial configuration $\mathbb{C}$ and LB$_1$ otherwise.
The time complexity of an algorithm to compute the lower bound can be $\mathcal{O}(B)$. As the second lower bound appears in the literature,
it is referred to as {\color{black}LB$_2$} by Tricoire et al. \cite{Tricoire-2018-COR}.
{\color{black}
LB$_2$  is initially proposed for unrestricted variant 13.
Similar to LB$_1$, it in fact can be applied to all the eight variants with individual moves, noting that any restricted variant can be naturally relaxed to an unrestricted variant.}

\noindent$\blacksquare$ \noindent \textbf{Revisit and Demonstration}: \\
Let {\color{black}$\mathbb{B}^2$} be the collection of blocks of the top $1^{st}$ layer.
Property~\ref{proty_lb2} can be expressed as: $f^{\text{BB}}(\mathbb{B}^2) \geq  1$ if $\mathbb{B}^2$ satisfies the condition of Property \ref{proty_lb2} ($\mathbb{B}^2$ satisfies P2 for short). Given the fact that $\mathbb{B}^2 \subseteq \mathbb{B}$, we have
\begin{align}
f(\mathbb{B}) \geq  & f^{\text{BG}}(\mathbb{B}) + f^{\overline{\text{BG}}}(\mathbb{B})
\geq  \sum_{b\in \mathbb{B}} f^{\text{BG}}(\{b\}) + f^{\text{BB}}(\mathbb{B}) \notag\\
\geq  & |\mathbb{B}^1| + f^{\text{BB}}(\mathbb{B}^2)
\geq  \begin{cases}
|\mathbb{B}^1| + 1, & \text{if $\mathbb{B}^2$ satisfies P2}\\
|\mathbb{B}^1|, & \text{otherwise}
\end{cases}\notag
\end{align}
which exactly gives LB$_2$ as a valid lower bound.

\noindent$\blacksquare$ \noindent \textbf{Illustration}:\\
For the instance displayed in Figure \ref{figure: two_instances} (a), we have {\color{black}$\mathbb{B}^2=\{4, 13, 6, 5\}$}.
The priorities of all four stacks are respectively {\color{black}4, 13, 2, 1} from the left to the right, and the lowest one is {\color{black}13}.
As the highest priority of blocks in $\mathbb{B}^2$ is {\color{black}4}, which is higher than {\color{black}13}.
Hence, $\mathbb{B}^2$ does not satisfy P2, and $\text{LB}_2 = \text{LB}_1 = {\color{black}5}$.
For the instance displayed in Figure \ref{figure: two_instances} (b), we have $\mathbb{B}^2=\{16, 17, 18, 19\}$.
The priorities of all four stacks are respectively 2, 1, 5, 4 from the left to the right, and the lowest one is 5.
The highest priority of blocks in $\mathbb{B}^2$ is 16, which is lower than 5.
Hence, $\mathbb{B}^2$ satisfies P2, and  $\text{LB}_2 = \text{LB}_1 + 1 = {\color{black}9 + 1 = 10}$.

\ \\
\noindent \textit{{\color{black}3) LB$_3$:  A Generalization of LB$_2$}}

\begin{property}
\label{proty_lb3}
At least one non-BG move has to be implemented on a block in each of the top $k$ layers,
if : (1) the target block is not in the top $k$ layers,
and (2) the highest priority of BP blocks in the top $k$ layers is lower than
the lowest priority of all stacks after removing the top $k-1$ layers.
\end{property}

%

\noindent$\blacksquare$ \noindent \textbf{Argument and Lower Bound Development}:\\
Condition (1) in Property \ref{proty_lb3} ensures that the target block remains unmoved until at least one block has been relocated from each of the top $k$ layers.
Hence, the first move of a block from each of the top $k$ layers is a relocation.
Moreover, a BP block (e.g., a block above the target block) exists among each of the top $k$ layers.
Condition (2) ensures that any BP block in one of the top $k$ layers remains BP after the first relocation implemented in that layer.
Therefore, the first move of a block in each of the top $k$ layers is a non-BG move.

Tricoire et al. \cite{Tricoire-2018-COR} introduce and analyze this property,
and propose a lower bound based on both Properties \ref{proty_lb1} and \ref{proty_lb3}.
Specifically, their lower bound is set to $\text{LB}_1 + k$, as long as the maximum top $k$ layers satisfy the conditions of Property~\ref{proty_lb3} in the initial configuration $\mathbb{C}$.
This lower bound is referred to as LB$_3$ as it probably is the third lower bound {\color{black}appears} in the literature.
Note that, LB$_3$ generalizes and dominates LB$_2$ \cite{Tricoire-2018-COR} since Property \ref{proty_lb3} generalizes Property \ref{proty_lb2}.

The time complexity of an algorithm to compute LB$_3$ can be $\mathcal{O}(B)$,
although the conditions of Property 3 have to be checked several times (i.e., for $k = 1, \ldots, K$, and $K < B / S$).
Our reasoning is as follows.
(\romannumeral1) The target block can be found with a time complexity of $\mathcal{O}(B)$.
Without loss of generality, let it {\color{black}be} below the top $K^{th}$ layer and there are at least $K$ blocks in each stack.
(\romannumeral2) The highest priority of BP blocks among the top $1^{st}$ layer can be computed with a time complexity of $\mathcal{O}(S)$,
and that of the top $k^{th}$ ($2 \leq k \leq K$) layers can be computed with a time complexity of $\mathcal{O}(S)$ based on the results of the top ${k-1}^{th}$,
i.e., $S$ comparison operations.
(\romannumeral3) The lowest priority of stacks after removing the top ${K-1}^{th}$ layers can be computed with a time complexity of $\mathcal{O}(B)$,
and the lowest priority of stacks after removing the top $k^{th}$ ($1 \leq k \leq K-2$) layers can be computed with a time complexity of $\mathcal{O}(S)$ based on the results of the top ${k+1}^{th}$,
i.e., $2S-1$ comparison operations.
(\romannumeral4) Given the above data,
the conditions of Property 3 can be checked with a time complexity of $\mathcal{O}(1)$ for any $k = 1, \cdots, K$.
Following the calculation $B+S+(K-1)S+B+(K-2)(2S-1) + K < 2B+3KS+2 < 5B+2$, we conclude the time complexity of the overall algorithm as $\mathcal{O}(B)$.
{\color{black}
Since its derivation is a direct extension of that of LB$_2$, LB$_3$ is applicable to all the eight variants with individual moves.}


\noindent$\blacksquare$ \noindent \textbf{Revisit and Demonstration}: \\
Let {\color{black}$\mathbb{B}^3$} be the collection of blocks in the top $k$ layers.
Property~\ref{proty_lb3} can be expressed as:
$f^{\overline{\text{BG}}}(\mathbb{B}^3) \geq  k$ if $\mathbb{B}^3$ satisfies the conditions of Property 3 ($\mathbb{B}^3$ satisfies P3 for short).
Given the fact that $\mathbb{B}^3 \subseteq \mathbb{B}$, we have
\begin{align}
f(\mathbb{B}) \geq  & f^{\text{BG}}(\mathbb{B}) + f^{\overline{\text{BG}}}(\mathbb{B})
\geq  \sum_{b\in \mathbb{B}} f^{\text{BG}}(\{b\}) + f^{\overline{\text{BG}}}(\mathbb{B}^3) \notag\\
\geq  & |\mathbb{B}^1| + f^{\overline{\text{BG}}}(\mathbb{B}^3)
\geq  |\mathbb{B}^1| + k.  \notag
\end{align}
Hence, if $\mathbb{B}^3$ is the maximum  top $k$ layers satisfying P3, this derivation exactly gives LB$_3$ as a valid lower bound.

\noindent$\blacksquare$ \noindent \textbf{Illustration}:\\
For the instance displayed in Figure \ref{figure: two_instances} (a), {\color{black}$\mathbb{B}^3 = \{4, 13, 6, 5\}$} for $k=1$.
The target block is 1 and is not in $\mathbb{B}^3$.
The priorities of all four stacks are respectively {\color{black}4, 13, 2, 1} from the left to the right, and the lowest one is {\color{black}13}.
The highest priority of BP blocks in $\mathbb{B}^3$ is {\color{black}4}, which is higher than {\color{black}13}.
Therefore, $\mathbb{B}^3$ does not satisfy P3 for $k = 1$.
Therefore, no $\mathbb{B}^3$ satisfies P3,
and $\text{LB}_3 = \text{LB}_1 + k = {\color{black}5 + 0 = 5}$.
For the instance displayed in Figure \ref{figure: two_instances} (b), $\mathbb{B}^3 = \{16, 17, 18, 19; 6, 14, 5, 4\}$ for $k = 2$.
The target block is 1 and is not in $\mathbb{B}^3$.
The priorities of all four stacks after removing the top 1 ($= 2 - 1$) layer of blocks are respectively 2, 1, 5, 4 from the left to the right, and the lowest one is 5.
The highest priority of BP blocks in $\mathbb{B}^3$ is 6, which is lower than 5. Therefore, $\mathbb{B}^3$ satisfies P3 for $k = 2$.
We can further evaluate a larger $\mathbb{B}^3$ by setting it to $ \{16, 17, 18, 19; 6, 14, 5, 4; 2, 12, {\color{black}9}, 8\}$, i.e.,  $k = 3$.
Again, the target block is 1 and is not in $\mathbb{B}^3$.
The priorities of all four stacks after removing the top 2 ($= 3 - 1$) layers of blocks are respectively 2, 1, {\color{black}9, 7} from the left to the right, and the lowest one is {\color{black}9}.
The highest priority of BP blocks in $\mathbb{B}^3$ is 6, which is higher than {\color{black}9}.
Hence, $\mathbb{B}^3$ does not satisfy P3 for $k=3$.
As a conclusion, we have the maximum $\mathbb{B}^3$ that satisfies P3 when $k = 2$, and $\text{LB}_3 = \text{LB}_1 + k = {\color{black}9 + 2 = 11}$.

\ \\
\noindent \textit{{\color{black}4) LB-N: Another Generalization of LB$_2$}}

\begin{property}
\label{proty_lb4}
Consider the initial configuration $\mathbb{C}$ where the target block is in stack $s$.
Perform an experiment by relocating each block above the target block once without considering the stack height limit or moving blocks in other stacks.

If some of the relocated block(s) cannot be transformed to be WP in any of such experiments,
we can conclude with respect to $\mathbb{C}$ that either (1) at least one BB move has to be implemented on one of the relocated blocks,
or (2) at least one GB or GG move has to be implemented on a block with the highest priority in one of the other $S - 1$ stacks.
\end{property}

%

\noindent$\blacksquare$ \noindent \textbf{Argument and Lower Bound Development}:\\
Since all the relocated blocks are BP in $\mathbb{C}$, i.e., above the target block,
the condition of Property 4 ensures that some of them have to be implemented with BB moves if the priorities of other $S - 1$ stacks are not lowered beforehand.
The priority of a stack can be lowered only if its block with the highest priority, i.e., a WP block, is retrieved or relocated.
Since the target block is below the relocated blocks,
no WP block can be retrieved before completely relocating those blocks.
Therefore, the priorities of other $S - 1$ stacks can only be lowered by relocating WP blocks,
i.e., conducting GB or GG moves.
In conclusion, at least one non BG move has to be implemented on one of the relocated blocks or a block with the highest priority in one of the other $S - 1$ stacks.

Tanaka and Mizuno \cite{Tanaka-2018-COR} introduce and analyze this property,
and propose a new lower bound based on both Properties \ref{proty_lb1} and \ref{proty_lb4}, which is referred to as LB-N.
They further have a slight modification by considering the stack height limit in a specific situation. As such a change is rather minor, we do not include it in the following discussions to minimize distractions.

Specifically,  LB-N is set to $\text{LB}_1 + 1$ if the condition of Property 4 is satisfied according to an iterative procedure, and LB$_1$ otherwise.
First, they check the condition of Property~\ref{proty_lb4} for the initial configuration.
If satisfied, they increase LB-N by 1 and terminate. Otherwise, they remove the target block as well as all blocks above it,
and recheck the condition of Property 4 for the updated configuration (and the updated target block).
They repeat the above procedure until $\text{LB-N} = \text{LB}_1 + 1$ or all the blocks are removed.
Note that, if a block is removed then all blocks above it are removed, which suggests that an initially BP (WP, respectively) block remains BP (WP, respectively) after the removal.
Hence, the validity of the aforementioned iterative procedure is guaranteed \cite{Tanaka-2018-COR}.

For a target block $i$, the time complexity of checking the condition of Property \ref{proty_lb4} can be $B_{i}\log{S}$,
where $B_{i}$ is the number of blocks above block $i$ at the beginning of the iteration if stacks are sorted in the increasing order of their priorities beforehand \cite{Tanaka-2018-COR}.
The $\log{S}$ comes from the dichotomy to select an appropriate stack for each relocated block.
However, the priority of the stack of the target block is changed after removing blocks during each iteration.
The time complexity of finding the new order of the stack is $\mathcal{O}(\log{S})$, since priorities of other stacks remain unchanged and dichotomy can be used.
Then stacks have to be resorted with a time complexity of $\mathcal{O}(S)$.
Therefore, the time complexity of the whole algorithm for LB-N is $\mathcal{O}(BS)$.

{\color{black}
LB-N initially is proposed for unrestricted variant 9, where it generalizes LB$_2$ \cite{Tanaka-2018-COR}.
Indeed, it can be directly applied to four variants with distinct priorities and individual moves.
For other four variants with duplicate priorities and individual moves,  we note that it could also be applied if its derivation can be modified with some minor changes.
For example, when multiple target blocks exist, we can simply keep one but remove other target blocks and the blocks above them from $\mathbb{C}$.
The updated configuration allows us to perform the operations presented in Property 4 to derive a lower bound to $\mathbb{C}$.
}

\noindent$\blacksquare$ \noindent \textbf{Revisit and Demonstration}: \\
During one iteration, let {\color{black}$\mathbb{B}^4$} be the collection of the blocks above the target block and a block with the highest priority in each of the other $S - 1$ stacks (if not empty).
Then, Property \ref{proty_lb4} can be expressed as: $f^{\overline{\text{BG}}} (\mathbb{B}^4) \geq  1$ if $\mathbb{B}^4$ satisfies the condition of Property \ref{proty_lb4} ($\mathbb{B}^4$ satisfies P4 for short).
Given the fact that $\mathbb{B}^4 \subseteq \mathbb{B}$, we have
\begin{align}
f(\mathbb{B}) \geq  & f^{\text{BG}}(\mathbb{B}) + f^{\overline{\text{BG}}}(\mathbb{B})
\geq  \sum_{b\in \mathbb{B}} f^{\text{BG}}(\{b\}) + f^{\overline{\text{BG}}}(\mathbb{B}^4) \notag\\
\geq  & |\mathbb{B}^1| + f^{\overline{\text{BG}}}(\mathbb{B}^4)
\geq  \begin{cases}
|\mathbb{B}^1| + 1, & \text{if $\mathbb{B}^4$ satisfies P4}\\
|\mathbb{B}^1|, & \text{otherwise}
\end{cases}\notag
\end{align}
which exactly gives LB-N as a valid lower bound.

\noindent$\blacksquare$ \noindent \textbf{Illustration}:\\
{\color{black}
For the instance displayed in Figure \ref{figure: two_instances} (a),
block 5 is above the target block, i.e., block 1, in stack 4, and
the blocks with the highest priority in each of the other three stacks are respectively blocks 4, 13 and 2 from the left to the right.
Hence, $\mathbb{B}^4 = \{5\} \cup \{4, 13, 2\}$.
Since block 5 can become WP if is relocated to stack 2,
$\mathbb{B}^4$ does not satisfies P4, and we update the configuration by removing blocks 1 and 5.
For this new configuration, block 2 becomes the target block and
blocks 6 and 7 are above it in stack 3.
The blocks with the highest priority in each of the other three stacks are respectively blocks 4, 13 and 3 from the left to the right.
Hence, $\mathbb{B}^4 = \{6, 7\} \cup \{4, 13, 3\}$.
Since one of blocks 6 and 7 cannot become  WP if both of them are relocated only once,
$\mathbb{B}^4$ satisfies P4, and we set $\text{LB-N} = \text{LB}_1 + 1 = 5 + 1 = 6$.}
For the instance displayed in Figure \ref{figure: two_instances} (b),
blocks 17, 14, 12 and 10 are above the target block in stack 2,
and the blocks with the highest priority in each of the other three stacks are respectively blocks 2, 5 and 4 from the left to the right.
{\color{black}Hence, $\mathbb{B}^4 = \{17, 14, 12, 10\} \cup \{2, 5, 4\}$.}
Since block 17 cannot be relocated once to be WP, $\mathbb{B}^4$ satisfies P4,
and we have $\text{LB-N} = \text{LB}_1 + 1 = {\color{black}9 + 1 = 10}$.

\subsection{New Structural Properties}
In this subsection, we present four new structural properties for some subsets of the initial configuration.
{\color{black} Note that existing properties (i.e., P2 and P3) heavily depend on the concept of ``layer'', i.e., a subset of $S$ blocks that are physically of the same depth in all stacks.
And only physically adjacent layers are considered together.
Different from this natural one, we introduce the new ``virtual layer'' concept, which consists of $S$ blocks such that each of them is from a unique stack, but they might be of different depths.
Moreover, multiple virtual layers that are not physically adjacent could be easily considered.
Clearly, for this very flexible and general structure, if some connections to the necessary relocation moves can be established, it definitely provides a fundamental and powerful tool to analyze the BRP.
To the best of our knowledge, no similar structure has been considered in any of prior research.}

\ \\
\noindent \textit{{\color{black}1) Property 5: An Insight from a Virtual Layer}}

\begin{theorem}
\label{thm_new5}
(Property 5)
Pick a block from each of the $S$ stacks to form a \textit{virtual layer}.
At least one non-BG move has to be implemented on blocks of this virtual layer,
if : (1) there exists a block piled below the virtual layer such that its priority is higher than the highest priority of blocks in the virtual layer,
and (2) the highest priority of BP blocks in the virtual layer is lower than the lowest priority of all stacks after removing blocks above the virtual layer.
\end{theorem}

\begin{nproof}
Consider the first move of a block in the virtual layer.
Without loss of generality, assume that it is implemented on block $i$.
We will prove that this move is a non-BG move.

When block $i$ is the first one in the virtual layer to be relocated,
other blocks in the virtual layer remain in their initial positions.
Then we consider the following two situations.
(\romannumeral1) If block $i$ is initially WP, it cannot be retrieved, since condition (1) ensures that a block with a higher priority has not been retrieved yet.
So, this first move must be a GB or GG move.
(\romannumeral2) Otherwise, if block $i$ is initially BP, it cannot be relocated to be WP,
since condition (2) ensures that any destination stack is with a higher priority.
Therefore, this first move is a BB move.
In conclusion, the first move in the virtual layer must be a non-BG move, and 
Theorem 3 is proved.
\end{nproof}

\noindent$\blacksquare$ \noindent \textbf{Illustration}:\\
Let {\color{black}$\mathbb{B}^5$} be a virtual layer.
For the instance displayed in Figure \ref{figure: two_instances} (b),
we can have $\mathbb{B}^5=\{2, 12, 18, 8\}$,
and the highest priority of its blocks is 2.
Note that block 1 is below the virtual layer and its priority is higher than 2.
The priorities of all four stacks after removing blocks above the virtual layer are respectively 2, 1, 5 and {\color{black}7} from the left to the right, and the lowest one is 7.
As the highest priority of BP blocks in $\mathbb{B}^5$ is 8 (lower than 7), $\mathbb{B}^5$ satisfies the conditions of Property 5 ($\mathbb{B}^5$ satisfies P5 for short). Hence, we have $f^{\overline{\text{BG}}}(\mathbb{B}^5) \geq  1$.


\ \\
\noindent \textit{{\color{black}2) Property 6: An Extension to Multiple Non-overlapping Virtual Layers}}\\[0.05in]
\indent Utilizing the general framework in Theorems  \ref{thm_LB_overall} and \ref{thm_LB_overall2}, we take advantage of Theorem \ref{thm_new5} to derive the following corollary.

\begin{corollary}
\label{cor_new6}
(Property 6)
Given $k$ non-overlapping virtual layers,
each of which satisfies P5,
then at least $k$ non-BG moves will be implemented on blocks of  those virtual layers.
\end{corollary}

\noindent$\blacksquare$ \textbf{Illustration}:\\
For the instance displayed in Figure \ref{figure: two_instances} (b),
we can have three non-overlapping virtual layers $\mathbb{B}^5_1 = \{16, 17, 18, 19\}$, $\mathbb{B}^5_2 = \{6, 14, 5, 4\}$, and {\color{black}$\mathbb{B}^5_3 = \{2, 12, 9, 7\}$},
which all satisfy P5.
Therefore, we can have $f^{\overline{\text{BG}}}(\mathbb{B}^5_1 + \mathbb{B}^5_2 + \mathbb{B}^5_3) \geq  3$.
Note in the illustration following Property \ref{proty_lb3} that, only 2 non-BG moves can be guaranteed if Property \ref{proty_lb3} is applied.

Unlike the conventional concept of top $k$ layers used in Property \ref{proty_lb3},
those virtual layers in Corollary \ref{cor_new6} are not necessarily to be top layers or piled consecutively.
Moreover, given the top $k$ layers that satisfies P3,
any virtual layer formed by blocks in the top $k$ layers satisfies P5.
Hence, it can be easily seen that Property \ref{proty_lb3} is a special case of Property 6.

\begin{corollary}
\label{cor_6generalizes3}
With the same configuration $\mathbb{C}$, Property 6 subsumes Property 3 as a special case.
\end{corollary}

{\color{black}
To identify a virtual layer satisfying P5, we have designed two algorithms both with $\mathcal{O}(BS)$ time complexity.
Details of the two algorithms, referred to as A5 and A5*, are presented in Appendix \ref{Appendix_Alg_P5}.
Clearly, they can always be reused to identify multiple non-overlapping virtual layers satisfying P6.
In addition, as shown in Appendix \ref{Appendix_Alg_P5}, our algorithms generalize the subset selection procedure of P3.
Note that the top layers satisfying P3 will also be identified by our algorithms, as they always start with physical layers from the top of $\mathbb{C}$.}


\ \\
\noindent \textit{{\color{black}3) Property 7: An Extension to Overlapping Virtual Layers}}\\[0.05in]
\indent Next, we extend to consider a more complex block subset where two virtual layers share an overlapped block. Actually, two virtual layers can cross once through the shared block so that  one layer does not need to be above the other layer.


\begin{theorem}
\label{thm_new7}
(Property 7)
{\color{black}Consider two virtual layers that both satisfy P5 and share exactly one WP block.
If the priority of the shared WP block is lower than the lowest priority of other $S - 1$ stacks (i.e., excluding the one with the shared WP block) after removing blocks above both of the two layers,}
then either at least 2 non-BG moves will be implemented on blocks in the two virtual layers,
or at least 1 GB and 1 BG moves will be implemented on the shared WP block.
\end{theorem}

\begin{nproof}
We consider the first move of a block in each of those two virtual layers.
Without loss of generality,
{\color{black}we assume that those moves are implemented on block $i_1$ in the  first layer and block $i_2$
in the second layer respectively,
and block $i_1$ is moved before block $i_2$.}

We consider the following two situations.
(\romannumeral1) If block $i_1$ is the shared WP block,
then the move on block $i_1$ is a GB move,
and a BG move is needed for block $i_1$ in the latter moves.
Therefore, at least 1 GB and 1 BG moves have to be implemented on the shared WP block.
(\romannumeral2) If block $i_1$ is not the shared WP block,
then the move on block $i_1$ is a non-BG move since the first virtual layer satisfies P5.
Since the shared WP block is not moved,
all blocks of the second virtual layer remain in their initial positions after the move of block $i_1$.
Then the move on block $i_2$ is also a non-BG move since the second virtual layer satisfies P5.
Therefore, at least 2 non-BG moves have to be implemented on blocks in those two virtual layers.
In conclusion, Theorem 4 is proved.
\end{nproof}

\noindent $\blacksquare$ \textbf{Illustration}:\\
Let {\color{black}$\mathbb{B}^7$} be two virtual layers sharing one WP block.
For the instance displayed in Figure \ref{figure: two_instances} (b),
we can have {\color{black}$\mathbb{B}^7=\{16, 17, \underline{5}, 19\} \cup \{6, 14, \underline{5}, 4\}$},
where the shared block is block 5, and the two virtual layers are $\mathbb{B}^5_1 = \{16, 17, 5, 19\}$ and $\mathbb{B}^5_2 = \{6, 14, 5, 4\}$.
Note that both the two virtual layers satisfy P5.
{\color{black}
The priorities of the other three stacks after removing blocks above both of the two virtual layers are respectively 2, 1 and 4 from the left to the right,
and the lowest one is 4, which is higher than the priority of the shared WP block.
}
Therefore, $\mathbb{B}^7$ satisfies the conditions of Property 7 ($\mathbb{B}^7$ satisfies P7 for short),
and we have $f^{\overline{\text{BG}}}(\mathbb{B}^7) \geq  2$.

{\color{black}
To identify a block subset $\mathbb{B}^7$ satisfying P7 for a given shared block, we have
designed an algorithm with $O(BS)$ time complexity. Details of this algorithm, referred to as A7, are presented in Appendix \ref{Appendix_Alg_P7}.}

\noindent \textbf{Remark:}\\
The consideration of BG moves has been shown in Property 1, which is rather straightforward
due to the initially BP blocks.
Nevertheless, it is worth pointing out that the possible BG move presented in Theorem \ref{thm_new7}
is actually for an initially WP block, which definitely is not obvious.
This new understanding, as well as {\color{black}the following Theorem \ref{thm_new8}} and the consequently strong lower bound, is obtained through a deeper analysis and a more involved reasoning on a particular structure.
Hence, we believe that, with the support from our general framework, a more comprehensive understanding on the BRP and stronger lower bounds can be expected by studying more sophisticated structures and deriving richer insights.

{\color{black}
\ \\
\noindent \textit{4) Property 8: A Forced Move due to Retrieval}\\[0.05in]
\indent
%
Different from those derived with layer concepts, in the following,  we introduce a new property that considers non-BG moves of some blocks, which is similar to Property 4. In fact, it is a generalization of Property 4.

\begin{theorem}
	\label{thm_new8}
	(Property 8)
	Consider the initial configuration $\mathbb{C}$ where a block (say block $i$) is piled in a stack (say stack $s$).
	First pick some block(s) from the above of block $i$, which are with priorities lower than that of block $i$, to form the first block subset.
	Then pick at most one WP block from each of the other $S - 1$ stacks, which is again with a priority lower than that of block $i$, to form the second block subset.
	
	Perform an experiment by first removing all the blocks except blocks in those two subsets from $\mathbb{C}$,
	and then relocating each block of the first subset once without considering the stack height limit or moving blocks of the second subset.
	
	If some of the relocated block(s) cannot be transformed to be WP in any of such experiments,
	we can conclude with respect to $\mathbb{C}$ that either
	(1) at least one BB move has to be implemented on blocks of the first subset,
	or (2) at least one GB or GG move has to be implemented on blocks of the second subset.
\end{theorem}

\begin{nproof}	
	It is clear that to retrieve block $i$, blocks of the first subset must be moved away from stack $s$.
	Since they are all with priorities lower than that of block $i$, they can only be relocated.
	The condition of Property 8 ensures that some of these block(s) have to be implemented with BB moves if blocks of the second subset are not moved beforehand.
	If blocks of the second subset are moved beforehand, then they can only be relocated, since all of them are also with priorities lower than that of block $i$, i.e., they should be retrieved after block $i$.
	As they are initially WP, they have to be implemented with GB or GG moves if moved beforehand.
	Therefore, at least one non-BG move has to be implemented on blocks of those two subsets.
	In conclusion, Theorem \ref{thm_new8} is proved.
\end{nproof}

\noindent
$\blacksquare$ \textbf{Illustration}:\\
Let $\mathbb{B}^8$ be a union of the two block subsets proposed in Property 8.
For the instance displayed in Figure \ref{figure: two_instances} (a),
blocks 6 and 7 are above block 2 in stack 3,
blocks 4 and 3 (with priorities lower than that of block 2) are respectively in stacks 1 and 4.
Hence, we can have $\mathbb{B}^8 = \{6, 7\} \cup \{4, 3\}$.
Since one of blocks 6 and 7 cannot become WP if both of them are relocated only once,
$\mathbb{B}^8$ satisfies the condition of Property 8 ($\mathbb{B}^8$ satisfies P8 for short),
and we have $f^{\overline{\text{BG}}}(\mathbb{B}^8) \geq  1$.
Moreover, blocks 11 and 12 are above block 8 in stack 1,
blocks 9 and 10 (with priorities lower than that of block 8) are respectively in stacks 3 and 4.
Hence, we can have another $\mathbb{B}^8 = \{11, 12\}\cup \{9, 10\}$.
Since one of blocks 11 and 12 cannot become WP if both of them are relocated only once,
this $\mathbb{B}^8$ also satisfies P8,
and we also have $f^{\overline{\text{BG}}}(\mathbb{B}^8) \geq  1$.

It is straightforward that Property 8 generalizes Property 4,
since a block subset satisfying P4 always satisfies P8, but not vice versa.
Therefore, we can have the following corollary.

\begin{corollary}
	\label{cor_8generalizes4}
	With the same configuration $\mathbb{C}$, Property 8 subsumes Property 4 as a special case.
\end{corollary}

To identify a block subset $\mathbb{B}^8$ satisfying P8 for a given block $i$, we have designed an algorithm with a time complexity of $O(B\log{S})$.
Details of this algorithm, referred to as A8, are presented in Appendix \ref{Appendix_Alg_P8}.
As proven in Appendix \ref{Appendix_Alg_P8}, if a block subset satisfying P4 is identified in the derivation of LB-N (otherwise LB-N trivially reduces to LB$_1$), at least one block subset satisfying P8 can be derived by A8.

We also would like to  note that a block subset satisfying P8 contains no less than $S$ blocks, and actually can always be reduced to contain exactly $S$ blocks (see Proposition A-2 presented in Appendix \ref{Appendix_Alg_P8}).
As argued in the next subsection, this observation could be useful in identifying more nontrivial block subsets to derive a stronger lower bound.
}

\subsection{A Stronger Lower Bound}

In this subsection, by fully making use of both existing and new properties,
we present a new lower bound that could be much stronger than all existing ones.
{\color{black}
Based on the general framework presented in Theorems \ref{thm_LB_overall} and \ref{thm_LB_overall2} and our demonstrations on interpreting previous lower bounds, it can be easily seen that different partitions of the complete block set $\mathbb{B}$ generally lead to different lower bounds.
Hence, deriving the strongest lower bound can be converted into a combinatorial optimization problem, assuming that we have an exact understanding on the necessary relocations of all types of block subsets.
Nevertheless, neither we have investigated all significant block subsets nor our understanding on necessary relocations is thorough. Under such a situation, we focus on four non-dominated properties, and design a fast heuristic procedure to
partition $\mathbb{B}$ into critical block subsets and to  compute a strong lower bound.

Specifically, Properties 1, 5/6, 7 and 8 will be explored and evaluated to design this new lower bound,
which is referred to as LB$_4$.
Note that Properties 2 and 3 are generalized by Property 5/6, and Property 4 is generalized by Property 8.
The basic idea of our procedure is to use Properties 5/6, 7 and 8 one by one following an appropriate order.
We first pick the maximum number of block subsets satisfying Property 7 using Algorithm A7,
since two relocations can be derived from a block subset containing $2S-1$ blocks,
less than the two block subsets containing $2S$ blocks.
We then pick the maximum number of block subsets satisfying Property 5 from the unpicked blocks using Algorithm A5 (or its modification A5*).
We finally pick the maximum number of block subsets satisfying Property 8 from the unpicked blocks using Algorithm A8.
In the end,  LB$_4$ is set to the sum of LB$_1$ and the number of necessary relocations arising from those block subsets.
Note that on top of Algorithm A5, we have a modified version A5* (see Appendix \ref{Appendix_Alg_P5}) to fully make use of the virtual layer concept. In our numerical study, we embed each of them within the overall algorithm for LB$_4$ to build two implementations. If their outputs are different, we simply take the larger one as the final LB$_4$.

The time complexity of the above procedure for LB$_4$ is $\mathcal{O}(B^2 S)$.
We reason it as follows.
(\romannumeral1) LB$_1$ is computed once,
Algorithm A5 (or A5*) is used at most $B / S$ times,
both A7 and A8 are used at most $B$ times.
(\romannumeral2) Their time complexities are respectively $\mathcal{O}(B)$, $\mathcal{O}(BS)$, $\mathcal{O}(BS)$ and less than $\mathcal{O}(B\log{S})$.
Following the calculation $B + B / S \times BS + B \times (BS + B\log{S}) = B^2 S + B^2 \log{S} + B^2 + B$,
we conclude the overall time complexity as $\mathcal{O}(B^2 S)$.

Similar to Properties 2 and 3 behind LB$_2$ and LB$_3$, Properties 5 and 7 can be directly applied to the eight BRP variants with individual moves.
Different from Property 4, Property 8 can also be directly applied without modification to those eight BRP variants, given that a target block is not specifically considered.
As a result, LB$_4$ can be directly applied to all the eight BRP variants with individual moves.
}

Based on Theorems \ref{thm_LB_overall} and \ref{thm_LB_overall2}, we next prove that LB$_4$ is a valid lower bound that dominates all existing ones.

\begin{theorem}
\label{thm_lb4}
LB$_4$ is a valid lower bound to $f(\mathbb{B})$, and it dominates LB$_1$, LB$_2$, LB$_3$ and LB-N.
\end{theorem}

{\color{black}
\begin{nproof}
Let $\mathbb{B}^1$ be the collection of BP blocks in the initial configuration.
Let $\mathbb{B}^7_1, \cdots, \mathbb{B}^7_l$ be the picked subsets satisfying P7,
and $b_1, \cdots, b_l$ be the corresponding shared WP blocks.
Let $\mathbb{B}^5_1, \cdots, \mathbb{B}^5_m$ be the picked subsets satisfying P5,
and $\mathbb{B}^8_1, \cdots, \mathbb{B}^8_n$ be the picked subsets satisfying P8.
Note that, $\mathbb{B}^7_1, \cdots, \mathbb{B}^7_l, \mathbb{B}^5_1, \cdots, \mathbb{B}^5_m, \mathbb{B}^8_1, \cdots, \mathbb{B}^8_n$ are not overlapping with each other. According to Theorems \ref{thm_LB_overall} and \ref{thm_LB_overall2}, we have
\begin{align}
f(\mathbb{B}) \geq  & f^{\text{BG}} \left(\mathbb{B}\right) +
f^{\overline{\text{BG}}} \left(\mathbb{B}\right)
\notag\\
\geq  & f^{\text{BG}} \left(\mathbb{B}^1 \cup \{b_1, \cdots, b_l\}\right) +
\notag\\
& f^{\overline{\text{BG}}} \left(\mathbb{B}^7_1 \cup \cdots \cup \mathbb{B}^7_l \cup \mathbb{B}^5_1 \cup \cdots \cup \mathbb{B}^5_m \cup \mathbb{B}^8_1 \cup \cdots \cup \mathbb{B}^8_n\right)
\notag\\
\geq  & f^{\text{BG}} \left(\mathbb{B}^1\right) +
\sum_{i=1}^l {f^{\text{BG}} \left(\{b_i\}\right)} +
\notag\\
& \sum_{i=1}^{l} {f^{\overline{\text{BG}}} \left(\mathbb{B}^7_i\right)} +
\sum_{i=1}^{m} {f^{\overline{\text{BG}}} \left(\mathbb{B}^5_i\right)} +
\sum_{i=1}^{n} {f^{\overline{\text{BG}}} \left(\mathbb{B}^8_i\right)}
\notag\\
= & f^{\text{BG}} \left(\mathbb{B}^1\right) +
\sum_{i=1}^l {\left( f^{\text{BG}} \left(\{b_i\}\right) + f^{\overline{\text{BG}}} \left(\mathbb{B}^7_i\right) \right)} +
\notag\\
& \sum_{i=1}^{m} {f^{\overline{\text{BG}}} \left(\mathbb{B}^5_i\right)} + \sum_{i=1}^{n} {f^{\overline{\text{BG}}} \left(\mathbb{B}^8_i\right)} \notag\\
\geq  & |\mathbb{B}^1| + 2l + m + n\notag
\end{align}
which exactly gives LB$_4$ as a valid lower bound.
Assume that the top $k$ physical layers are identified in the derivation of LB$_3$.
According to the discussion after Corollary \ref{cor_6generalizes3} (as well as that in Appendix \ref{Appendix_Alg_P5}),
they satisfy Property P5/P6 and will definitely be identified by Algorithm A5.
Hence, we have $m {\rm{\geq}}  k {\rm{-}} 2l$, and
\begin{align}
\text{LB}_4 \geq  & |\mathbb{B}^1| + 2l + m + n
\geq  |\mathbb{B}^1| + 2l + (k - 2l) + n \notag\\
\geq  &|\mathbb{B}^1| + k = \text{LB}_3 \geq  \text{LB}_2 \geq  \text{LB}_1. \notag
\end{align}
Similarly, assume a subset satisfying P4 is found in the derivation of LB-N.
According to the discussion after Corollary \ref{cor_8generalizes4} (as well as Proposition A-1 in Appendix \ref{Appendix_Alg_P8}),
at least a subset satisfying P8 can be derived by Algorithm A8.
Hence we have
\begin{align}
\text{LB}_4 \geq  & \begin{cases}
|\mathbb{B}^1| + 1 = \text{LB-N}, & \text{if $2l+m=0$}\\
|\mathbb{B}^1|+2l + m\geq |\mathbb{B}^1|+1 = \text{LB-N}, & \text{otherwise}.
\end{cases}\notag
\end{align}
In conclusion, Theorem \ref{thm_lb4} is proved.
\end{nproof}}

\noindent$\blacksquare$ \noindent \textbf{Illustration}:\\
For the instance displayed in Figure \ref{figure: two_instances} (a),
{\color{black}
no subset $\mathbb{B}^7$ or $\mathbb{B}^5$ is picked,
but two subsets $\mathbb{B}_1^8 = \{6, 7\} \cup \{4, 3\}$ and $\mathbb{B}_2^8 = \{11, 12\} \cup \{9, 10\}$ are picked.
Therefore, $\text{LB}_4 = 5 + 2 \times 0 + 0 + 2 = 7$.}
For the instance displayed in Figure \ref{figure: two_instances} (b),
a subset {\color{black}$\mathbb{B}^7 = \{16, 17, \underline{5}, 19\} \cup \{6, 14, \underline{5}, 4\}$} is first picked,
two subsets $\mathbb{B}^5_1 = \{2, 12, 18, 8\}$ and {\color{black}$\mathbb{B}^5_2 = \{3, 10, 9, 7\}$}) are subsequently picked,
and no subset $\mathbb{B}^8$ is picked.
Therefore, LB$_4$ = {\color{black}9} + 2 $\times$ 1 + 1 $\times$ 2 + 0 = {\color{black}13}.
Given that LB$_1$, LB$_2$, LB$_3$ and LB-N for this instance are respectively {\color{black}9, 10, 11 and 10}, it
indicates that LB$_4$ is much stronger.

{\color{black}
\noindent \textbf{Remarks:} \\
We would like to highlight that, under the general framework presented in Theorems \ref{thm_LB_overall} and \ref{thm_LB_overall2}, the development of a strong lower bound can be standardized into two steps.
$(i)$ We discover some new structural properties from a block subset, and derive an insight on the involved necessary relocations and their move types.
$(ii)$ Based on the pool of those properties (and the corresponding block subsets), including existing ones and newly discovered one(s), we design an algorithm to
partition the complete block set $\mathbb{B}$ and to determine the necessary relocations associated with block subsets, aiming to maximize the total number of relocations collected over those subsets.}

\section{Exact Computational Methods}
\label{section_formulation}
In addition to the derivations of the lower bounds on the number of relocations,
we present in this section a new MIP model of the BRP that can be directly computed by an MIP solver. After observing that some strong MIP relaxations of the BRP can be computed quickly,
we develop an MIP formulation based exact algorithm that can further improve our solution capability.
As shown in Section \ref{section_experiments},
comparing to the state-of-the-art formulation,
the two approaches have significantly better computational performances.

{\color{black}
As previously mentioned, a particular system might have additional concerns and requirements, and MIP formulations are actually flexible and general to handle them.
This advantage is illustrated in the next section where results developed in this section are extended to accommodate several practical considerations that often occur in container yards or steel slab yards.  Unless otherwise stated, the BRP is referred to variant 9, a rather standard one, in the remainder of this paper. }

\subsection{A New MIP Formulation of the BRP}
\label{subsection_formulation1}

Because of specifications of the BRP, we note that all existing MIP formulations define 0-1 variables for each block regarding its dynamic position(s) among stacks during the whole retrieval process \cite{Caserta-2012-EJOR,  Petering-2013-EJOR, Silva-2018-EJOR}.
Clearly, real instances will incur large numbers of binary variables, which cause these formulations difficult to compute.
In this paper, we introduce 0-1 variables to define the adjacency relationship between a pair of blocks and lift-up and lift-down operations involved in a relocation move, without considering stacks in the retrieval process.
With this strategy, the number of variables can be reduced significantly, and our formulation, as shown in Section~\ref{section_experiments}, is computationally much more friendly than the state-of-the-art one.
Next, we introduce necessary notations for sets and parameters to support our model development. Recall that
$\mathbb{B}:=\{1,\ldots,B\}$ and $\mathbb{S}:=\{1,\ldots,S\}$ have been introduced to represent the sets of blocks and stacks, respectively.


${\color{black}\mathbb{B}'}=\mathbb{B}\cup \{B\text{+}1\}$: the extended set of blocks, noting that block $B\text{+}1$, a virtual block, represents the floor, regardless of stacks.

$\mathbb{C}\in \{0,1\}^{B\times (B\text{+}1)}$: the matrix representing the initial configuration of blocks. Specifically, $C_{ij}=1$ if block $i$ is piled directly upon block $j$, and 0 otherwise.

$H$: the height limit (in terms of blocks) of stacks, i.e., the maximum number of blocks can be piled on a stack.

{\color{black}
$h_i$: the height of (i.e., the vertical position of) block $i$ in $\mathbb{C}$.}

$L$: the lower bound of the number of relocations.

$T$: the number of relocation turns, where a relocation turn (turn in short) includes a relocation move and all the subsequent retrieval moves before the next relocation move \cite{Silva-2018-EJOR}.
It is set as an upper bound of the number of relocations.

${\color{black}\mathbb{T}}= \{1, 2, ..., T\}$: the set of relocation turns, which naturally
partitions the complete retrieval process into $T$ stages. For simplicity, we also use $0$ to denote the initial stage before any relocation.

With those notations, our BRP problem can be precisely stated as follows.
\begin{problem}
\label{prob_BRP}
Given a configuration $\mathbb{C}$  with $B$ blocks (of distinct priorities) piled on $S$ stacks with height limit $H$, and a crane that moves one block at a time and retrieves blocks from 1 to B sequentially.
Determine a sequence of moves with the least number of relocations to retrieve all blocks.
\end{problem}

Next, we define decision variables and present the complete MIP formulation.
As mentioned, 0-1 variables are introduced to describe the adjacency relationship between a pair of blocks and lift-up and lift-down operations involved in every relocation turn.
To facilitate an easy understanding, constraints and their interpretations are presented groupwise based on their connections.

\noindent $\blacksquare$ \textbf{Variables}:

{\color{black}$x_{ij}^t$}: equals 1 if block $i$ is piled directly upon block $j$ at the end of turn $t$, and 0 otherwise.

{\color{black}$\hat{y}_{ij}^t$}: equals 1 if block $i$ is directly relocated from (i.e., lifted-up from) block $j$ during turn $t$, and 0 otherwise. Note that $i \neq j$.

{\color{black}$\check{y}_{ij}^t$}: equals 1 if block $i$ is directly relocated to (i.e., lifted-down to) and piled upon block $j$ during turn $t$, and 0 otherwise. Note that $i \neq j$.

{\color{black}$z_{ij}^t$}: equals 1 if block $i$ is readily retrieved from the top of block $j$ during turn $t$, and 0 otherwise. Note that $i < j$.

{\color{black}$u_i^t$}: {\color{black}the height of block $i$ just after the lift-down move of the relocation in turn $t$.}

\noindent$\blacksquare$ \textbf{Objective function}:
{\color{black}
\begin{align}
& \min {\sum_{i \in \mathbb{B}} {\sum_{j \in \mathbb{B}' \backslash \{i\}} {\sum_{t \in \mathbb{T}} {\check{y}_{ij}^t}}}}
\tag{o1}
\label{obj_1}
\end{align}}

\noindent$\blacksquare$ \textbf{Constraints}:\\[0.05in]
$(\romannumeral1)$ Initial and dynamic relationships between blocks.
\begin{align}
& x_{ij}^0 = C_{ij}
\ \
\forall i \in \mathbb{B};\
j \in \mathbb{B}' \backslash \{i\}
\tag{x1}
\label{cons_x_1}\\
& x_{ij}^t = x_{ij}^{t\text{-}1}- \hat{y}_{ij}^t + \check{y}_{ij}^t
\ \
\forall i \in \mathbb{B};\
j \in \mathbb{B}',\ j \mathrm{<} i;\
t \in \mathbb{T}
\tag{x2}
\label{cons_x_2}\\
& x_{ij}^t = x_{ij}^{t\text{-}1}- \hat{y}_{ij}^t + \check{y}_{ij}^t - z_{ij}^t
\
\forall i \in \mathbb{B};
j \in \mathbb{B}', j \mathrm{>} i;
t \in \mathbb{T}
\tag{x3}
\label{cons_x_3}\\
& x_{ij}^{T} = 0
\ \
\forall i \in \mathbb{B};\
j \in \mathbb{B}' \backslash \{i\}
\tag{x4}
\label{cons_x_4}\\
& x_{ij}^t \in \{0,1\}
\ \
\forall i \in \mathbb{B};\
j \in \mathbb{B}' \backslash \{i\};\
t \in \mathbb{T} \cup \{0\}
\tag{x5}
\label{cons_x_5}
\end{align}
This set of constraints defines the dynamic adjacency relationship, due to lift-up and lift-down operations and retrieval moves in a relocation turn, between blocks $i$ and $j$ during the retrieval process.\\[0.05in]
%
%
%
$(\romannumeral2)$ Restrictions on the lift-up operation per turn.
\begin{align}
& \sum_{i \in \mathbb{B}} {\sum_{j \in \mathbb{B}' \backslash \{i\}} {\hat{y}_{ij}^t}} = 1
\ \
\forall t \in \mathbb{T},\ t \leq L
\tag{$\hat{\text{y}}$1}
\label{cons_y_up_1}\\
& \sum_{i \in \mathbb{B}} {\sum_{j \in \mathbb{B}' \backslash \{i\}} {\hat{y}_{ij}^t}} \leq \sum_{i \in \mathbb{B}} {\sum_{j \in \mathbb{B}' \backslash \{i\}} {\hat{y}_{ij}^{t\text{-}1}}}
\ \
\forall t \in \mathbb{T},\ t \mathrm{>} L
\tag{$\hat{\text{y}}$2}
\label{cons_y_up_2}\\
& \hat{y}_{ij}^t \leq x_{ij}^{t\text{-}1}
\ \
\forall i \in \mathbb{B};\
j \in \mathbb{B}' \backslash \{i\};\
t \in \mathbb{T}
\tag{$\hat{\text{y}}$3}
\label{cons_y_up_3}\\
& \sum_{j \in \mathbb{B}' \backslash \{i\}} {\hat{y}_{ij}^t} \mathrm{\leq }
\sum_{j \in \mathbb{B}' \backslash \{i\}} {x_{ij}^{t\text{-}1}} \mathrm{-}
\sum_{j \in \mathbb{B} \backslash \{i\}} {x_{ji}^{t\text{-}1}}
\ \
\forall i \in \mathbb{B}; t \in \mathbb{T}
\tag{$\hat{\text{y}}$4}
\label{cons_y_up_4}\\
& \hat{y}_{ij}^t \in \{0,1\}
\ \
\forall i \in \mathbb{B};\
j \in \mathbb{B}' \backslash \{i\};\
t \in \mathbb{T}
\tag{$\hat{\text{y}}$5}
\label{cons_y_up_5}
\end{align}
Constraints in (\ref{cons_y_up_1})-(\ref{cons_y_up_2}) guarantee exactly one lift-up operation is performed per turn among the first $L$ turns, and no more than one can be done in any subsequent turn.
Also, \eqref{cons_y_up_2} suggests that, given that the number of necessary relocations could be less than $T$, empty turns will be arranged after actual relocation turns.
Constraints in (\ref{cons_y_up_3})-(\ref{cons_y_up_4}) ensure the feasibility of a lift-up operation using the block relationship from the previous turn.
Note that, as the virtual block representing the floor is introduced, the right-hand-side of (\ref{cons_y_up_4}) equals 0 if block $i$ is not the topmost one in a stack.
\\[0.05in]
%
%
%
$(\romannumeral3)$ Restrictions on the lift-down operation per turn.
\begin{align}
& \sum_{j \in \mathbb{B}' \backslash \{i\}} {\check{y}_{ij}^t} = \sum_{j \in \mathbb{B}' \backslash \{i\}} {\hat{y}_{ij}^t}\ \
\forall i \in \mathbb{B};\
t \in \mathbb{T}
\tag{$\check{\text{y}}$1}
\label{cons_y_dn_1}\\
& {\color{black}\sum_{j \in \mathbb{B} \backslash \{i\}} {\check{y}_{ji}^t} \leq \sum_{j \in \mathbb{B}' \backslash \{i\}} {x_{ij}^{t\text{-}1}} - \sum_{j \in \mathbb{B} \backslash \{i\}} {\hat{y}_{ji}^t}}\ \
{\color{black}\forall i \in \mathbb{B};
t \in \mathbb{T}}
\tag{$\check{\text{y}}$2}
\label{cons_y_dn_2}\\
& {\color{black}\sum_{j \in \mathbb{B}} {\check{y}_{j(B\text{+}1)}^t} \leq 1 - \sum_{j \in \mathbb{B}} {\hat{y}_{j(B\text{+}1)}^t}\ \
\forall t \in \mathbb{T}
\tag{$\check{\text{y}}$3}
\label{cons_y_dn_3}}\\
& \sum_{j \in \mathbb{B} \backslash \{i\}} {\check{y}_{ji}^t} \mathrm{\leq}
\sum_{j \in \mathbb{B}' \backslash \{i\}} {x_{ij}^{t\text{-}1}} \mathrm{-}
\sum_{j \in \mathbb{B} \backslash \{i\}} {x_{ji}^{t\text{-}1}}\ \
\forall i \in \mathbb{B}, t \in \mathbb{T}
\tag{$\check{\text{y}}$4}
\label{cons_y_dn_4}\\
& \sum_{j \in \mathbb{B}} {\check{y}_{j(B\text{+}1)}^t} \leq S - \sum_{j \in \mathbb{B}} {x_{j(B\text{+}1)}^{t\text{-}1}}\ \
\forall t \in \mathbb{T}
\tag{$\check{\text{y}}$5}
\label{cons_y_dn_5}\\
&  \check{y}_{ij}^t \in \{0,1\}\ \
\forall i \in \mathbb{B};\
j \in \mathbb{B}' \backslash \{i\};\
t \in \mathbb{T}
\tag{$\check{\text{y}}$6}
\label{cons_y_dn_6}
\end{align}
Constraints in {\color{black}(\ref{cons_y_dn_1})-(\ref{cons_y_dn_3})} guarantee that the lift-down and lift-up operations of a turn should be performed with the same block, but upon two different blocks.
Constraints in (\ref{cons_y_dn_4})-(\ref{cons_y_dn_5}) represent that a relocated block can only be piled on the top of a topmost block or directly on the floor.
\\[0.05in]
%
%
%
$(\romannumeral4)$ Restrictions on the retrieval moves per turn.
\begin{align}
& {\color{black} z_{ij}^t \leq x_{ij}^{t\text{-}1} - \hat{y}_{ij}^t + \check{y}_{ij}^t
\ \
\forall i \mathrm{\in} \mathbb{B};\
j \in \mathbb{B}',\ j > i;\
t \mathrm{\in} \mathbb{T}}
\tag{z1}
\label{cons_z_1}\\
& \sum_{j \in \mathbb{B}', j \mathrm{>} i} {z_{ij}^t} {\rm \leq}
\sum_{j \in \mathbb{B}' \backslash \{i\}} {x_{ij}^{t\text{-}1}} \text{--}
\sum_{j \in \mathbb{B}, j \mathrm{>} i} {(x_{ji}^{t\text{-}1} \text{--} \hat{y}_{ji}^t {\text{+}} \check{y}_{ji}^t)}
\
\forall i \mathrm{\in} \mathbb{B};
t \mathrm{\in} \mathbb{T}
\tag{z2}
\label{cons_z_2}\\
& \sum_{j \in \mathbb{B}', j \mathrm{>} i} {\sum_{\tau \in \mathbb{T}, \tau \leq t} {z_{ij}^{\tau}}} \mathrm{\leq }
\sum_{j \in \mathbb{B}', j \mathrm{>} i \text{-} 1} {\sum_{\tau \in \mathbb{T}, \tau \leq t} {z_{(i\text{-}1)j}^{\tau}}}
\
\forall i \mathrm{\in} \mathbb{B} \backslash \{1\};
t \mathrm{\in} \mathbb{T}
\tag{z3}
\label{cons_z_3}\\
& z_{ij}^t \in \{0,1\}\ \
\forall i \in \mathbb{B};\
j \in \mathbb{B}',\ j \mathrm{>} i;\
t \in \mathbb{T}
\tag{z4}
\label{cons_z_4}
\end{align}
{\color{black}Constraints in (\ref{cons_z_1})-(\ref{cons_z_2}) ensure that a retrieval move is implementable, i.e., a target block is not blocked by a lower prioritized block.}
Constraints in (\ref{cons_z_3}) guarantee that the prioritized retrieval list is followed throughout the retrieval process.
\\[0.05in]
%
%
%
$(\romannumeral5)$ Restrictions on the stack height limit per turn.
\begin{align}
& u_i^t \mathrm{\geq}
u_j^t \mathrm{+}
1 \mathrm{-}
{\color{black}H (1 \mathrm{-}
x_{ij}^{t\text{-}1} \mathrm{+}
\hat{y}_{ij}^t \mathrm{-}
\check{y}_{ij}^t)}
\
\forall i \mathrm{\in} \mathbb{B};
j \mathrm{\in} \mathbb{B} \backslash \{i\};
t \mathrm{\in} \mathbb{T}
\tag{u1}
\label{cons_u_1}\\
& u_i^t \in {\color{black}\{1,\dots, H\}} \hspace{0.4cm}
\forall i \in \mathbb{B};\
t \in \mathbb{T}
\tag{u2}
\label{cons_u_2}
\end{align}

Overall, our MIP formulation, which is referred to as $\text{BRP-m3}$ following the convention in the literature, is summarized as follows.
{\color{black}
\begin{align}
\text{BRP-m3}:\ & \text{(\ref{obj_1})}
\notag\\
s.t.\ \ & \text{(\ref{cons_x_1}), \ldots, (\ref{cons_x_5}),
(\ref{cons_y_up_1}) \ldots (\ref{cons_y_up_5}),
(\ref{cons_y_dn_1}) \ldots (\ref{cons_y_dn_6})}
\notag\\
& \text{(\ref{cons_z_1}) \ldots (\ref{cons_z_4}),
	(\ref{cons_u_1}), (\ref{cons_u_2})}
\notag
\end{align}

\noindent$\blacksquare$ \textbf{Model Modifications for Other Variants}:\\
With some minor modifications, BRP-m3 actually can handle 7 more variants out of 16 ones. We first modify it to consider the unrestricted BRP with duplicate priorities.


As a block's ID does not indicate its priority, we let $i_{\rm{p}}$ represent block $i$'s priority.
Also, let $\mathbb{B}_k =\{i\in \mathbb{B}: i_{\rm{p}}=k\}$, i.e., the set of blocks of priority $k$.
Clearly, a BRP variant with distinct priorities is a special case with  $|\mathbb{B}_k|=1$ for $1\leq k\leq B$.
Note that blocks' priorities only affect their retrieval moves, not other moves.
Hence, we only need to replace constraints in (\ref{cons_x_2})-(\ref{cons_x_3}) and (\ref{cons_z_1})-(\ref{cons_z_4}) of BRP-m3 by the following ones.
\begin{align}
& x_{ij}^t \mathrm{=} x_{ij}^{t\text{-}1} \mathrm{-} \hat{y}_{ij}^t \mathrm{+} \check{y}_{ij}^t
\hspace{0.62cm}
\forall i \in \mathbb{B};\
j \in \mathbb{B}',\ j_{\rm{p}} \mathrm{<} i_{\rm{p}};\
t \in \mathbb{T}
\tag{x2*}\\
& x_{ij}^t \mathrm{=} x_{ij}^{t\text{-}1} \mathrm{-} \hat{y}_{ij}^t \mathrm{+} \check{y}_{ij}^t \mathrm{-} z_{ij}^t
\ \
\forall i {\rm{\in}} \mathbb{B};
j {\rm{\in}} \mathbb{B}' \backslash \{i\}, j_{\rm{p}} \mathrm{\geq} i_{\rm{p}};
t {\rm{\in}} \mathbb{T}
\tag{x3*}\\
& z_{ij}^t \mathrm{\leq} x_{ij}^{t\text{-}1} \mathrm{-} \hat{y}_{ij}^t \mathrm{+} \check{y}_{ij}^t
\hspace{0.6cm}
\forall i \mathrm{\in} \mathbb{B};\
j \mathrm{\in} \mathbb{B}' \backslash \{i\},\ j_{\rm{p}} \mathrm{\geq} i_{\rm{p}};\
t \mathrm{\in} \mathbb{T}
\tag{z1*}\\
& \sum_{j \in \mathbb{B}' \backslash \{i\}, j_{\rm{p}} \mathrm{\geq} i_{\rm{p}}} {z_{ij}^t} {\rm \leq}
\sum_{j \in \mathbb{B}' \backslash \{i\}} {x_{ij}^{t\text{-}1}} \mathrm{-}
\sum_{j \in \mathbb{B}, j_{\rm{p}} \mathrm{>} i_{\rm{p}}} {(x_{ji}^{t\text{-}1} \mathrm{-} \hat{y}_{ji}^t \mathrm{+} \check{y}_{ji}^t)}
\notag\\
& \hspace{6.15cm} \forall i \mathrm{\in} \mathbb{B};
t \mathrm{\in} \mathbb{T}
\tag{z2*}\\
& |\mathbb{B}_{(i_{\rm{p}}\text{-}1)}| \mathrm{\times} \sum_{j \in\mathbb{B}' \backslash \{i\}, j_{\rm{p}} \mathrm{\geq} i_{\rm{p}}} {\sum_{\tau \in\mathbb{T}, \tau \mathrm{\leq} t} {z_{ij\tau}}} \mathrm{\leq}
\notag\\
& \sum_{k \in \mathbb{B}_{(i_{\rm{p}}\text{-}1)}} {\sum_{j \in\mathbb{B}' \backslash \{k\}, j_{\rm{p}} \mathrm{\geq} k_{\rm{p}}} {\sum_{\tau \in\mathbb{T}, \tau \mathrm{\leq} t} {z_{k j \tau}}}}
\
\forall i \mathrm{\in} \mathbb{B}, i_{\rm{p}} \mathrm{\geq} 2;
t \mathrm{\in} \mathbb{T}
\tag{z3*}\\
& z_{ij}^t \in \{0,1\}
\hspace{1.55cm}
\forall i \in \mathbb{B};\
j \in \mathbb{B}',\ j_{\rm{p}} \mathrm{\geq} i_{\rm{p}};\
t \in \mathbb{T}
\tag{z4*}
\end{align}

To consider restricted BRPs where only forced moves are allowed,
we introduce the following constraints to ensure that we will keep relocating blocks above a target block until it is retrieved.
\begin{align}
& \sum_{j \in \mathbb{B}' \backslash \{i\}} {\hat{y}_{ij}^t} \mathrm{\geq}
\sum_{j \in \mathbb{B} \backslash \{i\}} {\hat{y}_{ji}^{t\text{-}1}} \mathrm{-}
\sum_{j \in \mathbb{B}' \backslash \{i\}, j_{\rm{p}} \mathrm{\geq} i_{\rm{p}}} {z_{ij}^{t\text{-}1}}
\notag\\
& \hspace{4.92cm} \forall i \in \mathbb{B};\
t \in \mathbb{T} \backslash \{1\}
\tag{$\hat{y}$6}\\
& \hat{y}_{i(B+1)}^t = 0
\ \ \
\hspace{3.5cm} \forall i \in \mathbb{B};\
t \in \mathbb{T}
\tag{$\hat{y}$7}
\end{align}

Regarding a BRP variant with the incomplete retrieval, it can be equivalently converted to one with duplicate priorities and the complete retrieval, by assigning no-to-retrieve blocks with the same lowest priority and ignoring their retrieval moves in the resulting solution.
Overall, with the aforementioned discussions, it can be seen that BRP-m3 can be modified to solve  all eight BRP variants with individual moves.
Certainly, we recognize that those BRP-m3 based models might not be strong since variant-specific properties could be used to strengthened them to achieve a better computational performance.
As our focus is on variant 9 and on BRP-m3, we leave it as a future research direction.

\ \\
\noindent$\blacksquare$ \textbf{Model Simplifications}:\\
Next, we present a few rather simple modifications that can effectively reduce our MIP model's complexity, while do not hurt its correctness.
One is that a couple of groups of discrete variables are relaxed
into continuous ones.
Another one
is that some constraints are actually
dominated by others so that they can be eliminated. The last one is that a set of constraints can be included to ensure that block 1 is retrieved as soon as
possible, which decreases the size of feasible set.

\begin{proposition}
\label{prop_new_brm_m3}
BRP-m3 can be simplified without sacrificing its correctness by the following three modifications.\\	
$(a)$ Relax integer variables $x_{ij}^t$ and $u_i^t$ to be continuous ones;\\
$(b)$ Eliminate constraints in (\ref{cons_y_up_3}), (\ref{cons_z_1}) and (\ref{cons_z_2});\\
$(c)$ Augment BRP-m3 with  the following equalities.
\begin{align}
& z_{1i}^t \mathrm{=} \hat{y}_{j1}^t
\ \
\forall i \mathrm{\in} \mathbb{B}' {\rm\ and\ } C_{1i} \mathrm{=} 1;\
j \mathrm{\in} \mathbb{B} {\rm\ and\ } C_{j1} \mathrm{=} 1;\
t \mathrm{\in} \mathbb{T}
\tag{e1}
\label{cons_e_1}\\
& \sum_{t \in \mathbb{T}} {\hat{y}_{i1}^t} = 1
\ \
\forall i \in \mathbb{B} {\rm\ and\ } C_{i1} = 1
\tag{e2}
\label{cons_e_2}\\
& \check{y}_{i1}^t = 0
\ \
\forall i \in \mathbb{B} \backslash \{1\};\
t \in \mathbb{T}
\tag{e3}
\label{cons_e_3}\\
& x_{i1}^t =\hat{y}_{i1}^t = 0
\ \
\forall i \in \mathbb{B} \backslash \{1\} {\rm\ and\ } C_{i1} = 0;\
t \in \mathbb{T}
\tag{e4}
\label{cons_e_4}\\
& \hat{y}_{1i}^t = \check{y}_{1i}^t = 0
\ \
\forall i \in \mathbb{B}' \backslash \{1\};\
t \in \mathbb{T}
\tag{e5}
\label{cons_e_5}\\
& x_{1i}^t = z_{1i}^t = 0
\ \
\forall i \in \mathbb{B}' \backslash \{1\} {\rm\ and\ } C_{1i} = 0;\
t \in \mathbb{T}
\tag{e6}
\label{cons_e_6}\\
& u_1^t = h_1
\ \
\forall t \in \mathbb{T}
\tag{e7}
\label{cons_e_7}
\end{align}	
\end{proposition}

Detailed proofs are shown in Appendix \ref{Appendix_Proof-Proposition1}.
As those modifications can be made easily and are computationally effective, we adopt them as default in our study and refer BRP-m3 to the simplified MIP formulation in the remainder of this paper.}

\ \\
\noindent \textbf{Remarks:} \\
$(\romannumeral1)$
Note that parameter $T$ is needed in our formulation (as well as in all other MIP formulations {\color{black}for the unrestricted BRP}), while the minimum number of relocations is unknown beforehand.  In our numerical study,  we adopt the optimal value of the restricted BRP to derive this bound, which
has been utilized in the state-of-the-art formulation BRP-m2 \cite{Silva-2018-EJOR}.
Moreover, a stronger lower bound $L$ is preferred as shown in (\ref{cons_y_up_1}).
Hence, the new stronger lower bound presented in Section \ref{section_LB} can be directly applied.\\
$(\romannumeral2)$ One non-trivial issue associated with traditional MIP formulations is symmetricity, noting that stacks are identical.
This issue could incur a heavy computational burden for an MIP solver. Nevertheless, given that our variable definitions do not depend on stacks, this issue is naturally removed from our formulation.
Moreover, comparing to the state-of-the-art formulation BRP-m2, we note that our BRP-m3 is of a smaller size with less numbers  of discrete variables and inequalities.
{\color{black}
Let $\lambda = \min\{H,B - S + 1\} \times SBT$.
In BRP-m2, there are about $4\lambda$ discrete variables and $2\lambda$ inequalities, while the corresponding numbers in BRP-m3 are $2.5B^2T$ and $B^2T$ (or simply $4BT$ if the height limit is ignored).
Indeed, for the non-trivial BRP where $1<S< B$, we have $S(S-1)< B(S-1)$, which is equivalently to
$B< BS-S^2+S=(B - S + 1)S$.
Together with the fact that $HS>B$ (otherwise relocations are impossible),
it is clear that $\lambda$ is strictly greater than $B^2T$.
Hence, BRP-m3 is of a smaller size.
}
\\
$(\romannumeral3)$ When the stack height limit is negligible or not imposed, e.g., stacks of steel plates in a steel factory,
all variables $u_{it}$ and  constraints in (\ref{cons_u_1})-(\ref{cons_u_2}) can simply be eliminated from BRP-m3.
Indeed, we note that  $\text{BRP-m3}$ has a much stronger performance compared to existing ones under such a situation.

\subsection{An MIP Relaxation Based Iterative Procedure}
\label{subsection_formulation2}
We note in our study that some MIP formulations of the relaxed BRP problems have superior computational performances.
This observation inspires us to make use of those relaxations within an algorithmic framework to compute the original BRP.
To this end, we present a study that develops an MIP relaxation based iterative procedure to derive exact BRP solutions.
To the best of our knowledge, no similar algorithm design has been reported in the literature on the BRP.

Let a \textit{direct blockage}  be a blockage formed by two blocks piled in a way such that the lower prioritized one is directly on top of the higher prioritized one in the same stack.
For the instance displayed in Figure \ref{figure: two_instances} (a) (see Section \ref{section_LB}), block 7 and block 2 form a direct blockage, but block 6 and block 2 do not.
Unlike the original BRP that focuses on relocation moves, the next problem, which is a relaxation to the BRP, considers the number of direct blockages.
Recalling that $L$ is a lower bound to the BRP, we define the new problem as follows.

\begin{problem}
\label{prob_BRPR}
Given a configuration $\mathbb{C}$ with $B$ blocks (of distinct priorities) piled on $S$ stacks with height limit $H$,
and a crane that moves one block at a time and retrieves blocks from 1 to $B$ sequentially.
Determine a sequence of moves to retrieve all blocks in a way such that after its $L$ relocations (and all applicable retrieval moves), the sum of $L$ and the number of direct blockages in the resulting configuration is minimized.
\end{problem}

We next show that it is a relaxation to the original BRP.

\begin{theorem}
\label{thm_relaxation}
Problem \ref{prob_BRPR} is a relaxation of the BRP defined in Problem \ref{prob_BRP}.
\end{theorem}

\begin{nproof}
We prove it according to the following criteria:
(1) any feasible solution of Problem \ref{prob_BRP} is a feasible one to Problem \ref{prob_BRPR},
and (2) its {\color{black}objective} value with respect to Problem \ref{prob_BRPR} is less than or equal to that with respect to Problem \ref{prob_BRP}.

On one hand, we consider a feasible solution of the BRP with $U$ relocations and other retrieval moves. Since we have $L \leq U$, it is naturally feasible to Problem \ref{prob_BRPR}.

On the other hand, according to  \cite{Scholl-2018-IJPR},
the number of direct blockages in any configuration is a lower bound of the number of necessary relocations for that configuration.
So, for any  move sequence feasible to the BRP,
after implementing its first $L$ relocations and applicable retrieval moves,
the number of direct blockages in the resulting configuration is less than or equal to the number of relocations in the remaining sequence.
Therefore, the second criterion is satisfied.
In conclusion, Problem \ref{prob_BRPR} is a relaxation of the BRP.
\end{nproof}

In the following, we present an MIP formulation for Problem \ref{prob_BRPR}, which is referred to as BRP-m3R.
Because of its connection and similarity to Problem \ref{prob_BRP}, we re-use variables and constraints for model development.
Nevertheless, we highlight that $\mathbb{T}$ is defined with respect to $\{1, 2, ..., L\}$.
Also, to better describe our iterative procedure, we keep the constant $L$ in the objective function of this formulation.
{\color{black}
\begin{align}
\text{BRP-m3R}:\ & L+\min {\sum_{i \in \mathbb{B}} {\sum_{j \in \mathbb{B}, j < i} {x_{ij}^L}}}
\tag{o2}
\label{obj_2}\\
s.t.\ \ & \text{constraints of BRP-m3 $\backslash$ \{(\ref{cons_x_4}), (\ref{cons_y_up_2}), (\ref{cons_e_2})\}}
\notag
\end{align}}
%
%
%
\indent Next, we develop the following algorithm (in pseudo code) where BRP-m3R is computed and updated over and over to strengthen the lower bound of BRP-m3,
and finally produces the strongest lower bound,  i.e., the optimal value of  BRP-m3.

\noindent\rule[0.05\baselineskip]{3.5in}{0.5pt}
Algorithm IS: Basic Iterative Scheme\\
\noindent\rule[0.45\baselineskip]{3.5in}{0.5pt}

\begin{algorithmic}[1]


\State {$L \leftarrow 0$, $L' \leftarrow \text{a lower bound of BRP-m3}$}
\State {\textbf{while} $L < L'$}
\State {\hspace{0.5cm}$L \leftarrow L'$}
\State {\hspace{0.5cm}update set $\mathbb{T}$ in BRP-m3R and solve BRP-m3R}
\State {\hspace{0.5cm}$L' \leftarrow \text{the optimal value of \text{BRP-m3R}}$}
\State {\Return {an optimal solution of the last \text{BRP-m3R}}}
\end{algorithmic}
\noindent\rule[0.8\baselineskip]{3.5in}{0.5pt}

\begin{theorem}
\label{thm_relaxation}
The IS algorithm converges to an optimal solution of the BRP in a finite number of iterations.
\end{theorem}

\begin{nproof}
Note that in this iterative procedure, we first solve \text{BRP-m3R}, i.e., the formulation of Problem \ref{prob_BRPR}, with $L$ relocations to optimality.
If the optimal  value equals $L$, then we get a feasible solution of the BRP,
since all blocks, including those in the remaining configuration, have been
retrieved or are simply retrievable. As $L$ is a lower bound of the BRP,
the optimal value of the current \text{BRP-m3R} is that of   BRP-m3.
Otherwise, we increase $L$ to that optimal value, update \text{BRP-m3R}, and
then resolve it. Hence, the lower bound $L$  always increases before reaching optimality.  Given that the number of relocations in the BRP is finitely bounded (e.g., the closed-form upper bound of the BRP in \cite{Caserta-2012-EJOR}),  it follows naturally that the algorithm converges to an optimal solution in a finite number of iterations.
\end{nproof}

It is worth mentioning that although \text{BRP-m3R} will be computed possibly several times in this iterative scheme, the total solution time could be much less than that of  BRP-m3 if (i) \text{BRP-m3R} is easy to solve and (ii) the initial lower bound
$L$ is tight. Indeed, if the strong lower bound derived in Section \ref{section_LB} is applied, generally only a couple of iterations are needed.
Moreover, the aforementioned algorithm, referred to as IS, can be enhanced by a few simple techniques listed below.  As a result,
as demonstrated in Section~\ref{section_experiments}, our overall solution {\color{black}capability} can be further
improved.

\noindent \textbf{Remarks:}\\
Two fast heuristics to generate initial solutions for Problem \ref{prob_BRPR} at the beginning of each iteration have been designed to support a commercial solver with a fast computation.
Also, note that BRP-m3 (BRP-m3R, respectively) without {\color{black}the height limit} is a relaxation to that with {\color{black}the height limit}.
Given that the former is much easier to compute, a practical strategy is to derive its optimal solution and verify whether {\color{black}the height limit} is violated over relocation turns.
If not, that optimal solution is also optimal to the latter one (i.e., with {\color{black}the height limit}). Indeed, as shown in our computational experiments, this is the case among the majority of testing instances.
If violated, a reparation heuristic is designed to convert that solution into a feasible one, and might again be optimal.
The overall enhanced IS algorithm in pseudo code is presented in Appendix \ref{Appendix_Enhanced-Iterative-Scheme}.
Moreover, an upper bound of the number of relocations is not needed in our IS algorithms.

{\color{black}
\section{Industrial Considerations and Flexible Modeling}
\label{section_customizations}
Although we have introduced 4 major features to define 16 different BRP variants, an industrial system often has more practical restrictions or concerns due to its particular configurations, cost specifications, or working environment. For example, from a safety point of view, it is discouraged to have very heavy containers piled above light ones in a container yard. Also, in addition to
the number of relocations, practitioners care about the energy consumption in the movements \cite{Lopez-2019-CIE}, especially the vertical ones, due to containers' large weights.
As shown in this section, such practical factors and complexities actually can be flexibly captured by customizing  the standard MIP formulation BRP-m3. The resulting
formulations, which are again mixed integer programs, can be readily computed by any professional MIP solver. As professional solvers are public available and manipulating those solvers is relatively easy, the flexibility and general applicability of MIP formulations offer a great advantage to industrial practitioners to handle various concerns and requirements.


%
%
%
%
%
%
%

\subsection{BRP with Penalty Coefficients}
For 16 typical BRP variants, they treat all relocation moves equally, regardless the differences among blocks. Nevertheless,
to protect motors, bearings and ropes or chains of a crane, it is desired to reduce the relocation moves associated with heavy blocks. Moreover,
a common situation in a container yard is that the weight distribution inside one container might not be even, whose relocations could cause damages to equipment. Under those situations, practitioners would like to assign different penalty coefficients
to relocation moves such that  moves of some blocks (e.g., those aforementioned blocks) should be heavily penalized.
Note that it is similar to the non-uniform relocations discussed in \cite{Galle-2018-EJOR} where they build and compute an MIP model for the restricted BRP.

To reflect this consideration, we just need to modify the objective function of BRP-m3. Let $d_i$ be the penalty coefficient of a relocation move on block $i$.  The modified MIP formulation, denoted by
BRP-m3-PC, is
\begin{align}
\text{BRP-m3-PC}:\ & \min {\sum_{i \in \mathbb{B}}
{\left(d_i \sum_{j \in \mathbb{B}' \backslash \{i\}}
	{\sum_{t \in \mathbb{T}}
		{\check{y}_{ij}^t}}\right)}}
\notag\\
s.t.\ & \text{constraints of BRP-m3}\notag
\end{align}
%

\subsection{BRP Considering Energy Consumptions}

Moving heavy blocks consumes a large amount of gas or electricity, especially when performing the vertical lift-up and lift-down operations, that is costly.
So, many practitioners care not only the number of relocations, but also the energy consumption in the retrieval process.
To reflect this consideration in our modeling, we augment BRP-m3 with new variables, constraints, and a modified objective function.
Specifically, new constraints and variables are introduced to capture the movements of blocks, and the objective function is modified to jointly minimize those two terms in terms of their weighted sum.
As vertical moves (by the hoist) typically demand much more energy than horizontal ones (by the trolley on the bridge), we consider the former ones in this augmented formulation.

To facilitate our understanding,
Figure \ref{figure: vertical_movements} illustrates all vertical movements  of a block during 6 consecutive turns, where an upwards arrow indicates a lift-up operation and a downwards arrow indicates a lift-down operation.
According to this figure, this block is initially piled at tier 3 in turn 0, first relocated  to tier 2 in some other stack in turn 2, staying there in turn 3, relocated again to tier 4 in turn 4, staying there in turn 5,
and finally is retrieved in turn 6.
In total, there are three lift-up operations and three lift-down operations.
In Figure \ref{figure: vertical_movements}, two important facts are noted.
First, distances of the first lift-up and the last lift-down operations are predetermined by the block's initial position, which are independent of our relocation decisions.
Second, all other moves appear in pairs, i.e., a lift-down operation is always followed by a lift-up one, and those two operations involve the same vertical distance.
Hence, assuming that energy consumption is proportion to the moving distance, we just need to include paired moves and the associate distances in the augmented formulation.

\begin{figure}[!t]
\begin{center}
	\includegraphics[scale = 0.9]{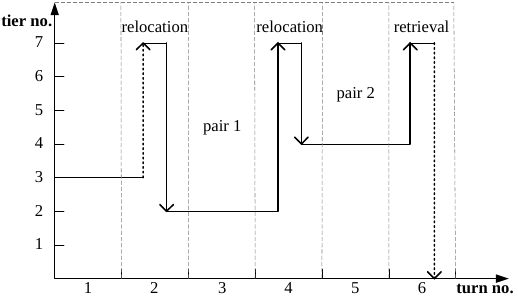}
	\caption{{\color{black}Vertical movements of a block during 6 turns}}
	\label{figure: vertical_movements}
\end{center}
\end{figure}

Let $v_i^t$ represents the vertical distance, in terms of the number of tiers, for the lift-down operation of block $i$ in turn $t$, and $c_i$ ($=\check{c}_i+\hat{c}_i$)  be the energy cost of one tier movement (where
$\check{c}_i$ and $\hat{c}_i$ are one tier costs of lift-down and lift-up movements on block $i$, respectively). Then, the objective function is updated as the following to represent the weighted sum over the number of relocations and the cost of
energy consumption.
\begin{align}
& \min \ {\alpha \sum_{i \in \mathbb{B}} {\sum_{j \in \mathbb{B}' \backslash \{i\}} {\sum_{t \in \mathbb{T}} {\check{y}_{ij}^t}}} +
\beta \sum_{i \in \mathbb{B}} {\sum_{t \in \mathbb{T}} c_iv_i^t}}
\tag{o3}
\label{obj_3}
\end{align}
where $0<\alpha<1$ is the weight parameter and $\beta=\frac{1-\alpha}{\max_{i \in \mathbb{B}} {c_i (H' - 1)}}$ normalizes all moving distances with $H'$ being the maximum height that a block can be lifted up to.
We next introduce a set of new constraints to define $v^t_i$.
\begin{align}
& v_i^t {\rm \geq} H' {\rm -} u_i^t {\rm -} (H' {\rm -} 1) \left(1 {\rm -} \sum_{j \in \mathbb{B}' \backslash \{i\}} {\check{y}_{ij}^t}\right)
\forall i \in \mathbb{B};
t \in \mathbb{T}
\tag{c1}
\label{cons_c_1}\\
& v_i^t \in [0, H'-1]\ \
\forall i \in \mathbb{B};\
t \in \mathbb{T}
\tag{c2}
\label{cons_c_2}
\end{align}

Note that, the value of $u_i^t$ in BRP-m3 is equal to or greater than the actual height of block $i$. Since it is unbounded from above, the constraints in (\ref{cons_c_1}) are very likely to be trivial.
To address this issue, we introduce the following constraints to force $u_i^t$ to be equal to the actual height of block $i$.
\begin{align}
& u_i^t \leq u_j^t + 1 + (H - 2) (1 - x_{ij}^{t\text{-}1} + \hat{y}_{ij}^t - \check{y}_{ij}^t)
\notag\\
& \hspace{3.9cm} \forall i \in \mathbb{B};\
j \in \mathbb{B} \backslash \{i\};\
t \in \mathbb{T}
\tag{c3}
\label{cons_c_3}\\
& u_i^t \leq 1 + (H - 1) (1 - x_{i(B+1)}^{t\text{-}1} + \hat{y}_{i(B+1)}^t - \check{y}_{i(B+1)}^t)
\notag\\
& \hspace{5.65cm} \forall i \in \mathbb{B};\
t \in \mathbb{T}
\tag{c4}
\label{cons_c_4}
\end{align}

Overall, the augmented formulation of BRP-m3 considering energy consumption,  referred to as BRP-m3-EC, is summarized as follows.
\begin{align}
\text{BRP-m3-EC}:\ & \text{(\ref{obj_3})}\notag\\
s.t.\ & \text{constraints of BRP-m3}
\notag\\
\ & \text{(\ref{cons_c_1}), \ldots, (\ref{cons_c_4})}
\notag
\end{align}

As for BRP-m3-PC and BRP-m3-EC, note that they derive optimal solutions (i.e., move sequences) with up to $T$ relocation turns. Unlike the standard BRP-m3 that simply minimizes the number of relocations, we would like to set parameter $T$ to a value larger than that of BRP-m3. By doing so, those formulations have a large search space
and probably lead to move sequences serving
our needs better. Certainly, practitioners can select an appropriate value for $T$ to avoid unreasonably many relocations.

\subsection{BRP Subject to Stacking Restrictions}
As mentioned, some stacking restrictions should be followed in the retrieval process \cite{Bruns-2016-EJOR}.
For example, very heavy containers should not be piled upon light ones in the container yard, or very long slabs should not be piled upon short ones in the slab yard.
Those restrictions actually are easily handled by adding some constraints to BRP-m3.

Let $\mathbb{B}_i^{\times}$ be the set of blocks upon which block $i$ is not allowed to pile. Then the modified formulation, referred to as BRP-m3-SR, is as follows.
\begin{align}
\text{BRP-m3-SR}:\ & \min {\sum_{i \in \mathbb{B}} {\sum_{j \in \mathbb{B}' \backslash \{i\}} {\sum_{t \in \mathbb{T}} {\check{y}_{ij}^t}}}}
\notag\\
s.t.\ & \text{constraints of BRP-m3} \notag\\
& x_{ij}^{t\text{-}1} - \hat{y}_{ij}^t + \check{y}_{ij}^t = 0\ \
\forall i \in \mathbb{B};\
j \in \mathbb{B}_i^{\times};\
t \in \mathbb{T} \notag
\end{align}
where $\mathbb{B}'$ is the extended block set (including the floor) defined in Section \ref{section_formulation}.

Due to the stacking restrictions, the minimum number of relocations of the new model could be larger than that of the standard BRP-m3, which suggests that a new upper bound is needed to set parameter $T$.
Note that it can be easily addressed, as any heuristic can be used to derive an upper bound.



\subsection{BRP Considering Retrieval Pace}
In some practical scenarios, a given retrieval pace of blocks should be satisfied. For example, in the steel factory, it is desirable to have the relocation moves and the retrieval moves mixed evenly. Otherwise, there could be just many relocation moves but no retrieval ones over a long period. Since no slab is delivered to the next stage in that period, a smooth production cannot be guaranteed \cite{Tang-2010-COR}. One solution is to limit the number of relocations before retrieving every block.
%

Like BRP-m3-SR, we just need to add some constraints to meet this requirement.
Let $T_i^{\text{max}}$ be the maximum number of relocations allowed before the retrieval of block $i$.  Then
the modified formulation, referred to as BRP-m3-RP, is as follows.
\begin{align}
\text{BRP-m3-RP}:\ & \min {\sum_{i \in \mathbb{B}} {\sum_{j \in \mathbb{B}' \backslash \{i\}} {\sum_{t \in \mathbb{T}} {\check{y}_{ij}^t}}}}
\notag\\
s.t.\ & \text{constraints of BRP-m3}
\notag\\
\ & \sum_{j \in \mathbb{B}', j > i}\ {\sum_{\tau \leq T_i^{\text{max}}} {z_{ij}^{\tau}}} = 1\ \ \forall i \in \mathbb{B}
\notag
\end{align}
Similar to BRP-m3-SR, the minimum number of relocations might be larger than that of BRP-m3. Again, a simple heuristic can be designed to derive an upper bound to set parameter $T$.
}


\section{Computational Results}
\label{section_experiments}

In this section, we report  our numerical results of the lower bounds, MIP formulations, and the IS algorithms.
Following the tradition in \cite{Caserta-2012-EJOR,Petering-2013-EJOR,Silva-2018-EJOR}, our test bed includes 13 groups, with 40 instances per group, from an instance set generated in \cite{Caserta-2011-ORS}.
Note that those instances are represented in
``a-b'' format with ``a'' denoting the current stack height, i.e., the number of blocks per stack, and ``b'' denoting the number of stacks.
Also, for each instance, we consider two situations, where there is no height limit and the height limit is set to $a+2$ \cite{Tricoire-2018-COR}.

%
All experiments are carried out on a desktop computer with the Windows 10 Professional 64-bit operating system, 32 GB RAM, and an Intel Core i7 7700 CPU with four 3.6-4.2 GHz cores and eight threads. %
Algorithms are implemented in C++ using CPLEX 12.61 and compiled with the Visual Studio 2013 C++ compiler.
For {\color{black}all instances}, the time limit is set to 3,600 seconds, the number of threads used is set to {\color{black}8}, and other parameters are in default settings. When the time limit is reached before obtaining an optimal solution, we set the solution time to 3,600 seconds in our report.

\subsection{Strength of Lower Bounds}
In this subsection, we compare our new lower bound, i.e., LB$_4$, with respect to existing ones appearing in the literature,
i.e., LB$_1$, LB$_2$, LB$_3$ and LB-N.
To be fair, we remove all retrievable blocks in the initial configurations before computing.
Otherwise, the strength of LB$_2$ and LB$_3$ will be weakened.

The first comparison  is displayed in  Figure \ref{figure: results_compare-LB-WH-rGap},
which reports the average relative gaps between the lower bounds and the actual optimal values,
i.e., the average of $(\textrm{opt-LB})/\textrm{opt}$ with ``opt'' being the optimal value and ``LB'' be one of the aforementioned lower bounds.
{\color{black}
We mention that the optimal values are computed by running the codes downloaded from \url{https://sites.google.com/site/shunjitanaka/brp} developed by Tanaka and Mizuno \cite{Tanaka-2018-COR}.}
The numerical results clearly confirm that our LB$_4$ often has much smaller relative gaps,
and strictly dominates all existing lower bounds for the test bed.
Moreover, for instances with larger height limits and less stacks,
the dominance of LB$_4$ is more significant.

{\color{black}
\begin{figure}[!t]
\begin{center}
	\includegraphics[scale = 0.9]{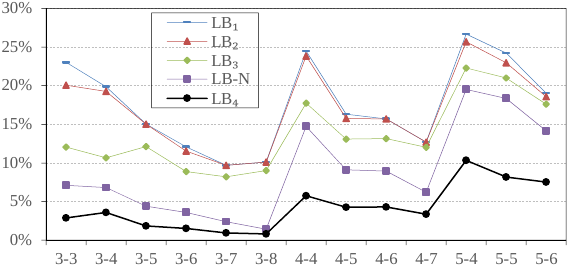}
	\caption{Average relative optimality gaps on instances with height limits}
	\label{figure: results_compare-LB-WH-rGap}
\end{center}
\end{figure}}

Our second comparison is on  the largest differences between the lower bounds and the actual optimal values in every group,
i.e., $\textrm{opt}-\textrm{LB}$, as shown in Figure \ref{figure: results_compare-LB-WH-aGap}.
Obviously, the overall trend in Figure \ref{figure: results_compare-LB-WH-aGap} largely agrees with that in Figure \ref{figure: results_compare-LB-WH-rGap},
showing LB$_4$ outperforms other lower bounds.
We highlight two more points.
The first one is that existing lower bounds are actually close to each other and demonstrate similar patterns. %
Nevertheless, LB$_4$ could be very different from them.
Note that LB$_4$ could be 4 relocations smaller than LB$_1$ in group 5-4, while other lower bounds are not more than 2 relocations than LB$_1$.
Another one is that,  in the most of worst cases, LB$_4$ is less than the optimal value by just a couple of relocations.
This observation is critical to our IS algorithms, which indicates that their convergences generally can be achieved in only a couple of iterations if LB$_4$ is adopted for initializations.

\begin{figure}[!t]
\begin{center}
\includegraphics[scale = 0.9]{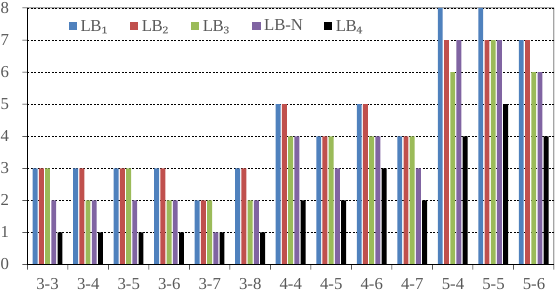}
\caption{Maximum absolute optimality gaps on instances with height limits}
\label{figure: results_compare-LB-WH-aGap}
\end{center}
\end{figure}

In Figure \ref{figure: results_compare-LB-WH-pOpt}, we finally present and  compare the percentages of instances over which the lower bounds are equal to the optimal values.
Again, our LB$_4$ has a clearly better performance over all existing lower bounds over all instances. In particular, for instances of small scales, our LB$_4$ is very likely  to be the optimal value.
Certainly, with {\color{black}larger height limits and more stacks}, such possibility becomes smaller, which can be explained by the increasing complexity of larger instances.

\begin{figure}[!t]
\begin{center}
\includegraphics[scale = 0.9]{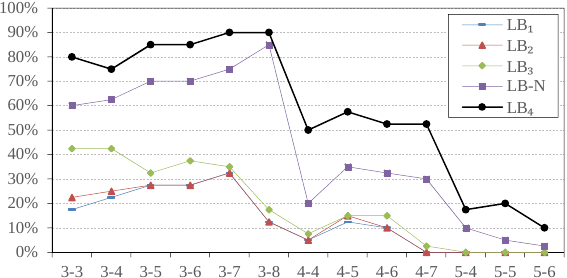}
\caption{Optimality percentages on instances with height limits}
\label{figure: results_compare-LB-WH-pOpt}
\end{center}
\end{figure}

\subsection{Computational Results Without a Height Limit}
In this subsection, we present the performance data of our computational methods, and benchmark with those of BRP-m2, which has been shown  to be the most effective one in the literature \cite{Silva-2018-EJOR}.
For BRP-m2, rather than trivially setting the height limit to $B$, we set it to $B - S+1$, which can be easily argued that spreading irretrievable  blocks among stacks is always preferred.

Table \ref{table: results_compare-m3-m2-NH} provides and compares CPLEX performance data of MIP formulations BRP-m2 (i.e., m2) and BRP-m3 (i.e., m3).
Columns ``\#Feasible'' and ``\#Optimal'' give the numbers of instances with feasible and optimal solutions before the time limit, respectively.
Column ``Time(s)'' represents the average solution time in seconds on all instances of each group.
As we use the time limit as the solution time when an instance cannot be solved, the average solution time could be misleading and biased against BRP-m3.
To do a fair comparison, we include two more columns ``Time*(s)'' and ``\#Nodes*'', which provide the average solution time and the average number of B\&B nodes over instances solved to optimality by BRP-m2.

\begin{table}[!t]
\centering
\caption{Comparison of BRP-m2 and BRP-m3 Without a Height Limit}
\label{table: results_compare-m3-m2-NH}%
\renewcommand\tabcolsep{2pt}
\begin{tabular}{ccrrrrrrccrccccc}
\hline
\multirow{2}[4]{*}{Case} &      & \multicolumn{2}{c}{\#Feasible} &      & \multicolumn{2}{c}{\#Optimal} &      & \multicolumn{2}{c}{Time(s)} &      & \multicolumn{2}{c}{Time*(s)} &      & \multicolumn{2}{c}{\#Nodes*} \bigstrut\\
\cline{3-4}\cline{6-7}\cline{9-10}\cline{12-13}\cline{15-16}         &      & \multicolumn{1}{c}{m2} & \multicolumn{1}{c}{m3} &      & \multicolumn{1}{c}{m2} & \multicolumn{1}{c}{m3} &      & m2   & m3   &      & m2   & m3   &      & m2   & m3 \bigstrut\\
\hline
3-3  &      & 40   & 40   &      & 40   & 40   &      & \multicolumn{1}{r}{2.2 } & \multicolumn{1}{r}{0.2 } &      & \multicolumn{1}{r}{2.1 } & \multicolumn{1}{r}{0.2 } &      & \multicolumn{1}{r}{149 } & \multicolumn{1}{r}{7 } \bigstrut[t]\\
3-4  &      & 40   & 40   &      & 40   & 40   &      & \multicolumn{1}{r}{169.8 } & \multicolumn{1}{r}{1.4 } &      & \multicolumn{1}{r}{168.6 } & \multicolumn{1}{r}{1.3 } &      & \multicolumn{1}{r}{2046 } & \multicolumn{1}{r}{168 } \\
3-5  &      & 37   & 40   &      & 34   & 40   &      & \multicolumn{1}{r}{893.0 } & \multicolumn{1}{r}{21.2 } &      & \multicolumn{1}{r}{416.7 } & \multicolumn{1}{r}{1.3 } &      & \multicolumn{1}{r}{3333 } & \multicolumn{1}{r}{65 } \\
3-6  &      & 27   & 40   &      & 19   & 39   &      & \multicolumn{1}{r}{2181.8 } & \multicolumn{1}{r}{129.7 } &      & \multicolumn{1}{r}{583.7 } & \multicolumn{1}{r}{1.4 } &      & \multicolumn{1}{r}{2757 } & \multicolumn{1}{r}{24 } \\
3-7  &      & 24   & 40   &      & 11   & 38   &      & \multicolumn{1}{r}{2851.6 } & \multicolumn{1}{r}{213.4 } &      & \multicolumn{1}{r}{878.1 } & \multicolumn{1}{r}{1.5 } &      & \multicolumn{1}{r}{3924 } & \multicolumn{1}{r}{1 } \\
3-8  &      & 6    & 40   &      & 0    & 35   &      & \multicolumn{1}{r}{3600.0 } & \multicolumn{1}{r}{565.1 } &      & -    & -    &      & -    & - \\
4-4  &      & 27   & 40   &      & 19   & 39   &      & \multicolumn{1}{r}{2297.5 } & \multicolumn{1}{r}{328.9 } &      & \multicolumn{1}{r}{857.7 } & \multicolumn{1}{r}{4.0 } &      & \multicolumn{1}{r}{3813 } & \multicolumn{1}{r}{118 } \\
4-5  &      & 7    & 38   &      & 3    & 28   &      & \multicolumn{1}{r}{3463.6 } & \multicolumn{1}{r}{1322.0 } &      & \multicolumn{1}{r}{1777.7 } & \multicolumn{1}{r}{1.6 } &      & \multicolumn{1}{r}{3425 } & \multicolumn{1}{r}{0 } \\
4-6  &      & 3    & 34   &      & 2    & 18   &      & \multicolumn{1}{r}{3491.1 } & \multicolumn{1}{r}{2290.8 } &      & \multicolumn{1}{r}{1416.2 } & \multicolumn{1}{r}{7.3 } &      & \multicolumn{1}{r}{1707 } & \multicolumn{1}{r}{38 } \\
4-7  &      & 0    & 26   &      & 0    & 10   &      & \multicolumn{1}{r}{3600.0 } & \multicolumn{1}{r}{2922.9 } &      & -    & -    &      & -    & - \\
5-4  &      & 5    & 31   &      & 3    & 14   &      & \multicolumn{1}{r}{3364.5 } & \multicolumn{1}{r}{2541.9 } &      & \multicolumn{1}{r}{456.3 } & \multicolumn{1}{r}{1.1 } &      & \multicolumn{1}{r}{1049 } & \multicolumn{1}{r}{0 } \\
5-5  &      & 0    & 13   &      & 0    & 2    &      & \multicolumn{1}{r}{3600.0 } & \multicolumn{1}{r}{3434.7 } &      & -    & -    &      & -    & - \\
5-6  &      & 0    & 4    &      & 0    & 2    &      & \multicolumn{1}{r}{3600.0 } & \multicolumn{1}{r}{3483.5 } &      & -    & -    &      & -    & - \\
sum  &      & 216  & 426  &      & 171  & 345  &      & -    & -    &      & -    & -    &      & -    & - \bigstrut[b]\\
\hline
\end{tabular}%
\end{table}%

Based on the results in Table \ref{table: results_compare-m3-m2-NH}, we note that our new MIP formulation BRP-m3 has a superior computational performance over the known best one BRP-m2.
We can find that, much more instances can be solved to feasibility or optimality using formulation BRP-m3 than using BRP-m2, especially for difficult instance groups.
Regarding the computational time, our new formulation has a drastically stronger power. Especially for instances that can be solved by BRP-m2, it is often the case that BRP-m3 solves to optimality a few hundred times quicker, with averagely two orders magnitude faster than BRP-m2.
A similar comparison can be found in the numbers of B\&B nodes.
{\color{black}
It is worth pointing out that for the six instances in 4-5 and 5-4 exactly solved by BRP-m2, our BRP-m3 model generates an optimal solution without any B\&B operation, while BRP-m2 averagely involves 2,237 B\&B nodes.}
Indeed, we observe that a large portion of instances in each group can be solved without any B\&B operation.
Hence, we believe that BRP-m3 is fundamentally different from existing ones, and is very close to the {\color{black}ideal formulation \cite{Wolsey-1998-Book}} of the BRP.

In Table \ref{table: results_compare-all-formulations-NH} we benchmark computational performances of IS algorithms, including the basic (i.e., IS) and the enhanced (i.e., IS*) implementations, with respect to BRP-m2 and BRP-m3.
{\color{black}Column ``\#Iters'' gives the average number of iterations over instances solved to optimality by each of the two IS algorithms.
Note that, for IS*, the number of iteration is set to zero when a heuristic directly produces an optimal solution. Hence, the average number of iterations  of IS* could  be less than 1.}
As can be seen, although those IS algorithms are iterative procedures, they perform similar or better than BRP-m3, noting that {\color{black} respectively 101\% and 108\%} more instances have been solved to optimality than BRP-m2.
Especially for the challenging instances in group 4-7 and 5-5, the enhanced IS* clearly outperforms regular MIP formulations by solving significantly more instances.
{\color{black} As for the number of iterations, it can be seen that those IS algorithms generally terminate in just a couple of iterations. Nevertheless, for some difficult instances, more iterations (up to 6) have been involved that demand a lot of computational time and degrade the overall performances.}

\begin{table}[!t]
	\centering
	\caption{Comparison of All Four Methods Without a Height Limit}
	\label{table: results_compare-all-formulations-NH}%
	\renewcommand\tabcolsep{2.7pt}	
	\begin{tabular}{crrrrrrccccccc}
		\hline
		\multirow{2}[4]{*}{Case} &      & \multicolumn{4}{c}{\#Optimal}  &      & \multicolumn{4}{c}{Time(s)} &      & \multicolumn{2}{c}{{\color{black}\#Iters}} \bigstrut\\
		\cline{3-6}\cline{8-11}\cline{13-14}         &      & \multicolumn{1}{c}{m2} & \multicolumn{1}{c}{m3} & \multicolumn{1}{c}{IS} & \multicolumn{1}{c}{IS*} &      & m2   & m3   & IS   & IS*  &      & IS   & IS* \bigstrut\\
		\hline
		3-3  &      & 40   & 40   & 40   & 40   &      & \multicolumn{1}{r}{2.2 } & \multicolumn{1}{r}{0.2 } & \multicolumn{1}{r}{0.3 } & \multicolumn{1}{r}{0.1 } &      & \multicolumn{1}{r}{1.2 } & \multicolumn{1}{r}{0.3 } \bigstrut[t]\\
		3-4  &      & 40   & 40   & 40   & 40   &      & \multicolumn{1}{r}{169.8 } & \multicolumn{1}{r}{1.4 } & \multicolumn{1}{r}{1.9 } & \multicolumn{1}{r}{1.5 } &      & \multicolumn{1}{r}{1.2 } & \multicolumn{1}{r}{0.5 } \\
		3-5  &      & 34   & 40   & 40   & 40   &      & \multicolumn{1}{r}{893.0 } & \multicolumn{1}{r}{21.2 } & \multicolumn{1}{r}{29.8 } & \multicolumn{1}{r}{47.5 } &      & \multicolumn{1}{r}{1.2 } & \multicolumn{1}{r}{0.3 } \\
		3-6  &      & 19   & 39   & 39   & 39   &      & \multicolumn{1}{r}{2181.8 } & \multicolumn{1}{r}{129.7 } & \multicolumn{1}{r}{151.6 } & \multicolumn{1}{r}{204.6 } &      & \multicolumn{1}{r}{1.1 } & \multicolumn{1}{r}{0.2 } \\
		3-7  &      & 11   & 38   & 38   & 39   &      & \multicolumn{1}{r}{2851.6 } & \multicolumn{1}{r}{213.4 } & \multicolumn{1}{r}{232.9 } & \multicolumn{1}{r}{170.7 } &      & \multicolumn{1}{r}{1.1 } & \multicolumn{1}{r}{0.2 } \\
		3-8  &      & 0    & 35   & 35   & 36   &      & \multicolumn{1}{r}{3600.0 } & \multicolumn{1}{r}{565.1 } & \multicolumn{1}{r}{543.2 } & \multicolumn{1}{r}{448.4 } &      & \multicolumn{1}{r}{1.0 } & \multicolumn{1}{r}{0.3 } \\
		4-4  &      & 19   & 39   & 37   & 37   &      & \multicolumn{1}{r}{2297.5 } & \multicolumn{1}{r}{328.9 } & \multicolumn{1}{r}{372.7 } & \multicolumn{1}{r}{320.4 } &      & \multicolumn{1}{r}{1.4 } & \multicolumn{1}{r}{0.8 } \\
		4-5  &      & 3    & 28   & 26   & 28   &      & \multicolumn{1}{r}{3463.6 } & \multicolumn{1}{r}{1322.0 } & \multicolumn{1}{r}{1410.5 } & \multicolumn{1}{r}{1234.2 } &      & \multicolumn{1}{r}{1.1 } & \multicolumn{1}{r}{0.6 } \\
		4-6  &      & 2    & 18   & 20   & 21   &      & \multicolumn{1}{r}{3491.1 } & \multicolumn{1}{r}{2290.8 } & \multicolumn{1}{r}{2165.7 } & \multicolumn{1}{r}{1763.2 } &      & \multicolumn{1}{r}{1.1 } & \multicolumn{1}{r}{0.4 } \\
		4-7  &      & 0    & 10   & 10   & 14   &      & \multicolumn{1}{r}{3600.0 } & \multicolumn{1}{r}{2922.9 } & \multicolumn{1}{r}{3173.3 } & \multicolumn{1}{r}{2517.0 } &      & \multicolumn{1}{r}{1.1 } & \multicolumn{1}{r}{0.5 } \\
		5-4  &      & 3    & 14   & 13   & 14   &      & \multicolumn{1}{r}{3364.5 } & \multicolumn{1}{r}{2541.9 } & \multicolumn{1}{r}{2595.2 } & \multicolumn{1}{r}{2482.2 } &      & \multicolumn{1}{r}{1.5 } & \multicolumn{1}{r}{1.2 } \\
		5-5  &      & 0    & 2    & 4    & 6    &      & \multicolumn{1}{r}{3600.0 } & \multicolumn{1}{r}{3434.7 } & \multicolumn{1}{r}{3334.0 } & \multicolumn{1}{r}{3167.4 } &      & \multicolumn{1}{r}{1.0 } & \multicolumn{1}{r}{0.3 } \\
		5-6  &      & 0    & 2    & 1    & 2    &      & \multicolumn{1}{r}{3600.0 } & \multicolumn{1}{r}{3483.5 } & \multicolumn{1}{r}{3560.3 } & \multicolumn{1}{r}{3464.9 } &      & \multicolumn{1}{r}{1.0 } & \multicolumn{1}{r}{0.5 } \\
		sum  &      & 171  & 345  & 343  & 356  &      & -    & -    & -    & -    &      & -    & - \bigstrut[b]\\
		\hline
	\end{tabular}%
\end{table}%

\subsection{Computational Results With Height Limits}
In this subsection, we present and analyze the performances of our computational methods on instances with {\color{black}height limits}.
Similar to Table \ref{table: results_compare-m3-m2-NH}, we first provide and compare in Table \ref{table: results_compare-m3-m2-WH} CPLEX performance data of MIP formulations BRP-m2 (i.e., m2) and BRP-m3 (i.e., m3).

Based on Table \ref{table: results_compare-m3-m2-WH}, we note again that BRP-m3 has a superior performance over BRP-m2.
Generally, on the instances with height limits, BRP-m3 is able to compute {\color{black}50\% and 64\%} more with feasible and optimal solutions respectively over BRP-m2.
Also, for instances exactly solved by BRP-m2, it is common that BRP-m3 solves to optimality {\color{black}4 to 	200} times quicker, with averagely {\color{black}50} times faster than BRP-m2.
A similar comparison can be found in the numbers of B\&B nodes.

\begin{table}[!t]
	\centering
	\caption{Comparison of BRP-m2 and BRP-m3 With Height Limits}
	\label{table: results_compare-m3-m2-WH}%
	\renewcommand\tabcolsep{2pt}
	\begin{tabular}{ccrrrrrrccrccccc}
		\hline
		\multirow{2}[4]{*}{Case} &      & \multicolumn{2}{c}{\#Feasible} &      & \multicolumn{2}{c}{\#Optimal} &      & \multicolumn{2}{c}{Time(s)} &      & \multicolumn{2}{c}{Time*(s)} &      & \multicolumn{2}{c}{\#Nodes*} \bigstrut\\
		\cline{3-4}\cline{6-7}\cline{9-10}\cline{12-13}\cline{15-16}         &      & \multicolumn{1}{c}{m2} & \multicolumn{1}{c}{m3} &      & \multicolumn{1}{c}{m2} & \multicolumn{1}{c}{m3} &      & m2   & m3   &      & m2   & m3   &      & m2   & m3 \bigstrut\\
		\hline
		3-3  &      & 40   & 40   &      & 40   & 40   &      & \multicolumn{1}{r}{1.2 } & \multicolumn{1}{r}{0.2 } &      & \multicolumn{1}{r}{1.2 } & \multicolumn{1}{r}{0.2 } &      & \multicolumn{1}{r}{83 } & \multicolumn{1}{r}{8 } \bigstrut[t]\\
		3-4  &      & 40   & 40   &      & 40   & 40   &      & \multicolumn{1}{r}{64.8 } & \multicolumn{1}{r}{3.5 } &      & \multicolumn{1}{r}{64.4 } & \multicolumn{1}{r}{3.4 } &      & \multicolumn{1}{r}{1861 } & \multicolumn{1}{r}{174 } \\
		3-5  &      & 40   & 40   &      & 38   & 40   &      & \multicolumn{1}{r}{455.2 } & \multicolumn{1}{r}{46.5 } &      & \multicolumn{1}{r}{290.0 } & \multicolumn{1}{r}{5.4 } &      & \multicolumn{1}{r}{3602 } & \multicolumn{1}{r}{79 } \\
		3-6  &      & 35   & 40   &      & 25   & 39   &      & \multicolumn{1}{r}{1755.2 } & \multicolumn{1}{r}{162.6 } &      & \multicolumn{1}{r}{623.2 } & \multicolumn{1}{r}{8.3 } &      & \multicolumn{1}{r}{3773 } & \multicolumn{1}{r}{69 } \\
		3-7  &      & 30   & 40   &      & 17   & 39   &      & \multicolumn{1}{r}{2482.6 } & \multicolumn{1}{r}{212.6 } &      & \multicolumn{1}{r}{809.2 } & \multicolumn{1}{r}{4.0 } &      & \multicolumn{1}{r}{3924 } & \multicolumn{1}{r}{0 } \\
		3-8  &      & 14   & 39   &      & 3    & 35   &      & \multicolumn{1}{r}{3562.3 } & \multicolumn{1}{r}{718.3 } &      & \multicolumn{1}{r}{3164.6 } & \multicolumn{1}{r}{22.4 } &      & \multicolumn{1}{r}{6261 } & \multicolumn{1}{r}{59 } \\
		4-4  &      & 31   & 40   &      & 25   & 37   &      & \multicolumn{1}{r}{1762.0 } & \multicolumn{1}{r}{515.7 } &      & \multicolumn{1}{r}{659.1 } & \multicolumn{1}{r}{21.8 } &      & \multicolumn{1}{r}{5548 } & \multicolumn{1}{r}{343 } \\
		4-5  &      & 13   & 35   &      & 6    & 21   &      & \multicolumn{1}{r}{3235.0 } & \multicolumn{1}{r}{1992.7 } &      & \multicolumn{1}{r}{1165.9 } & \multicolumn{1}{r}{10.8 } &      & \multicolumn{1}{r}{5040 } & \multicolumn{1}{r}{35 } \\
		4-6  &      & 6    & 30   &      & 2    & 15   &      & \multicolumn{1}{r}{3456.7 } & \multicolumn{1}{r}{2480.7 } &      & \multicolumn{1}{r}{730.8 } & \multicolumn{1}{r}{16.6 } &      & \multicolumn{1}{r}{2262 } & \multicolumn{1}{r}{49 } \\
		4-7  &      & 1    & 14   &      & 0    & 7    &      & \multicolumn{1}{r}{3600.0 } & \multicolumn{1}{r}{3155.2 } &      & -    & -    &      & -    & - \\
		5-4  &      & 8    & 22   &      & 3    & 11   &      & \multicolumn{1}{r}{3338.7 } & \multicolumn{1}{r}{2697.3 } &      & \multicolumn{1}{r}{114.2 } & \multicolumn{1}{r}{4.4 } &      & \multicolumn{1}{r}{919 } & \multicolumn{1}{r}{0 } \\
		5-5  &      & 0    & 8    &      & 0    & 2    &      & \multicolumn{1}{r}{3600.0 } & \multicolumn{1}{r}{3448.0 } &      & -    & -    &      & -    & - \\
		5-6  &      & 0    & 0    &      & 0    & 0    &      & \multicolumn{1}{r}{3600.0 } & \multicolumn{1}{r}{3600.0 } &      & -    & -    &      & -    & - \\
		sum  &      & 258  & 388  &      & 199  & 326  &      & -    & -    &      & -    & -    &      & -    & - \bigstrut[b]\\
		\hline
	\end{tabular}%
\end{table}

In Table \ref{table: results_compare-all-formulations-WH} we benchmark IS algorithms with respect to BRP-m2 and BRP-m3.
On difficult instances that BRP-m2 {\color{black}performs} very poorly, e.g., those in group {\color{black}3-8, 4-5, $\cdots$, 5-6}, the enhanced IS* has a {\color{black}clear} advantage.
Overall, it is able to solve {\color{black}73\%} more instances to optimality over BRP-m2.
It is worth noting that although the basic IS performs {\color{black}slightly poorer} than  BRP-m3, the enhanced IS* is {\color{black}much} better than BRP-m3.
Hence, it verifies the benefits of including enhancement techniques on improving our solution capability.
{\color{black}Specifically, for the 356 instances solved to optimality ignoring height limits, just 52 of them require reparation and 13 of them are successfully repaired.
Again, just a couple of iterations are needed for the IS algorithms.}

\begin{table}[!t]
	\centering
	\caption{Comparison of All Four Methods With Height Limits}
	\label{table: results_compare-all-formulations-WH}%
	\renewcommand\tabcolsep{2.7pt}	
	\begin{tabular}{crrrrrrccccccc}
		\hline
		\multirow{2}[4]{*}{Case} &      & \multicolumn{4}{c}{\#Optimal}  &      & \multicolumn{4}{c}{Time(s)} &      & \multicolumn{2}{c}{{\color{black}\#Iters}} \bigstrut\\
		\cline{3-6}\cline{8-11}\cline{13-14}         &      & \multicolumn{1}{c}{m2} & \multicolumn{1}{c}{m3} & \multicolumn{1}{c}{IS} & \multicolumn{1}{c}{IS*} &      & m2   & m3   & IS   & IS*  &      & IS   & IS* \bigstrut\\
		\hline
		3-3  &      & 40   & 40   & 40   & 40   &      & \multicolumn{1}{r}{1.2 } & \multicolumn{1}{r}{0.2 } & \multicolumn{1}{r}{0.4 } & \multicolumn{1}{r}{0.1 } &      & \multicolumn{1}{r}{1.2 } & \multicolumn{1}{r}{0.3 } \bigstrut[t]\\
		3-4  &      & 40   & 40   & 40   & 40   &      & \multicolumn{1}{r}{64.8 } & \multicolumn{1}{r}{3.5 } & \multicolumn{1}{r}{6.0 } & \multicolumn{1}{r}{1.6 } &      & \multicolumn{1}{r}{1.2 } & \multicolumn{1}{r}{0.5 } \\
		3-5  &      & 38   & 40   & 40   & 40   &      & \multicolumn{1}{r}{455.2 } & \multicolumn{1}{r}{46.5 } & \multicolumn{1}{r}{58.1 } & \multicolumn{1}{r}{32.2 } &      & \multicolumn{1}{r}{1.2 } & \multicolumn{1}{r}{0.3 } \\
		3-6  &      & 25   & 39   & 39   & 39   &      & \multicolumn{1}{r}{1755.2 } & \multicolumn{1}{r}{162.6 } & \multicolumn{1}{r}{172.2 } & \multicolumn{1}{r}{190.5 } &      & \multicolumn{1}{r}{1.1 } & \multicolumn{1}{r}{0.3 } \\
		3-7  &      & 17   & 39   & 37   & 39   &      & \multicolumn{1}{r}{2482.6 } & \multicolumn{1}{r}{212.6 } & \multicolumn{1}{r}{513.4 } & \multicolumn{1}{r}{212.3 } &      & \multicolumn{1}{r}{1.1 } & \multicolumn{1}{r}{0.3 } \\
		3-8  &      & 3    & 35   & 34   & 35   &      & \multicolumn{1}{r}{3562.3 } & \multicolumn{1}{r}{718.3 } & \multicolumn{1}{r}{748.6 } & \multicolumn{1}{r}{514.7 } &      & \multicolumn{1}{r}{1.0 } & \multicolumn{1}{r}{0.4 } \\
		4-4  &      & 25   & 37   & 34   & 36   &      & \multicolumn{1}{r}{1762.0 } & \multicolumn{1}{r}{515.7 } & \multicolumn{1}{r}{732.7 } & \multicolumn{1}{r}{425.7 } &      & \multicolumn{1}{r}{1.4 } & \multicolumn{1}{r}{1.0 } \\
		4-5  &      & 6    & 21   & 22   & 27   &      & \multicolumn{1}{r}{3235.0 } & \multicolumn{1}{r}{1992.7 } & \multicolumn{1}{r}{1889.7 } & \multicolumn{1}{r}{1573.4 } &      & \multicolumn{1}{r}{1.0 } & \multicolumn{1}{r}{0.8 } \\
		4-6  &      & 2    & 15   & 15   & 19   &      & \multicolumn{1}{r}{3456.7 } & \multicolumn{1}{r}{2480.7 } & \multicolumn{1}{r}{2426.6 } & \multicolumn{1}{r}{2029.5 } &      & \multicolumn{1}{r}{1.2 } & \multicolumn{1}{r}{0.5 } \\
		4-7  &      & 0    & 7    & 4    & 12   &      & \multicolumn{1}{r}{3600.0 } & \multicolumn{1}{r}{3155.2 } & \multicolumn{1}{r}{3319.6 } & \multicolumn{1}{r}{2652.7 } &      & \multicolumn{1}{r}{1.0 } & \multicolumn{1}{r}{0.4 } \\
		5-4  &      & 3    & 11   & 9    & 13   &      & \multicolumn{1}{r}{3338.7 } & \multicolumn{1}{r}{2697.3 } & \multicolumn{1}{r}{2871.1 } & \multicolumn{1}{r}{2565.3 } &      & \multicolumn{1}{r}{1.6 } & \multicolumn{1}{r}{1.2 } \\
		5-5  &      & 0    & 2    & 2    & 3    &      & \multicolumn{1}{r}{3600.0 } & \multicolumn{1}{r}{3448.0 } & \multicolumn{1}{r}{3454.3 } & \multicolumn{1}{r}{3330.1 } &      & \multicolumn{1}{r}{1.0 } & \multicolumn{1}{r}{0.0 } \\
		5-6  &      & 0    & 0    & 0    & 1    &      & \multicolumn{1}{r}{3600.0 } & \multicolumn{1}{r}{3600.0 } & \multicolumn{1}{r}{3600.0 } & \multicolumn{1}{r}{3552.8 } &      & -    & \multicolumn{1}{r}{1.0 } \\
		sum  &      & 199  & 326  & 316  & 344  &      & -    & -    & -    & -    &      & -    & - \bigstrut[b]\\
		\hline
	\end{tabular}%
\end{table}%

\subsection{Analysis on Influence of the Height Limit}
In this subsection, we analyze the influence of {\color{black}the height limit} on the performances of different computational methods.
The numbers of instances solved to optimality are shown in Figure \ref{figure: comparison_of_height_limit}.
From this figure, we note that the instances without {\color{black}a height limit} are actually more challenging to {\color{black}solve} for BRP-m2.
One explanation is that more binary variables have to be introduced to describe the possible stack height during the retrieval process, resulting in a significantly increased dimensionality.
As our modeling approach only needs variables describing the relationship between each pair of blocks, it naturally avoids that issue.
Another observation is that although the height limit has a non-trivial impact on the basic IS algorithm, enhancement techniques can actually largely reduce that impact, rendering the enhanced IS algorithm the most robust solution method.

\begin{figure}[!t]
	\begin{center}
		\includegraphics[scale = 0.9]{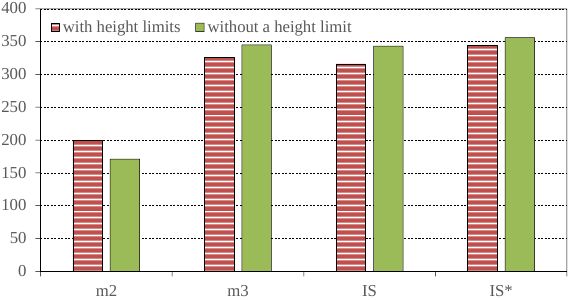}
		\caption{Number of instances solved to optimality}
		\label{figure: comparison_of_height_limit}
	\end{center}
\end{figure}

{\color{black}
\subsection{Computational Results of Customized MIP Formulations}

In this subsection, we present and analyze the performances of our customized formulations with additional industrial considerations on modified instances.

Specifically, on the basic instances presented above, we set the penalty coefficients $d_i$ as follows to generate testing instances.
\begin{align}
d_i = \begin{cases}
1,\ \ \text{if $b \leq \lceil 1 / 2 B \rceil$}\\
2,\ \ \text{otherwise.}
\end{cases}
\notag
\end{align}
To generate instances considering energy cost,
we set the cost coefficients $c_i = 3$ for all $i \in \mathbb{B}$,  the maximum lift-up height $H' = H + 1$,
and the weight parameter $\alpha=0.7$.  Also,
for both BRP-m3-PC and BRP-m3-EC, parameter $T$ is set to a value by rounding up 133\% of that in BRP-m3.

Also, to generate instances with stacking restrictions, we set $\mathbb{B}_i^{\times}$, i.e., the set of blocks upon which block $i$ is not allowed to pile, as follows.
\begin{align}
\mathbb{B}_i^{\times} = \begin{cases}
\text{blocks above block $i$}, & \text{if block $i$ is in stack 1 in $\mathbb{C}$}\\
\emptyset, & \text{otherwise.}
\end{cases}
\notag
\end{align}
Finally, to generate instances considering retrieval pace, we set $T_i^{\text{max}}$, i.e., the maximum number of relocations allowed before the retrieval of block $i$, as follows.
\begin{align}
T_i^{\text{max}} = \begin{cases}
h(i) + 1, & \text{if $b \leq \lceil 1 / 2 B \rceil$}\\
T, & \text{otherwise}
\notag\\
\end{cases}
\end{align}
where $h(i)$ returns the relocation turn in which block $i$ is retrieved using a simple heuristic.

\begin{table*}[!t]
	\centering
	\caption{{\color{black}Results of Customized Formulations Without a Height Limit}}
	\label{table: results_m3-IC-NH}
	\renewcommand\tabcolsep{4pt}
	\begin{tabular}{crrrrrrrrrrrrrcccccrccccc}
		\hline
		\multirow{2}[4]{*}{Case} &      & \multicolumn{5}{c}{\#Feasible}   &      & \multicolumn{5}{c}{\#Optimal}    &      & \multicolumn{5}{c}{Time(s)}      &      & \multicolumn{5}{c}{\#Nodes*} \bigstrut\\
		\cline{3-7}\cline{9-13}\cline{15-19}\cline{21-25}         &      & \multicolumn{1}{c}{BA} & \multicolumn{1}{c}{PC} & \multicolumn{1}{c}{EC} & \multicolumn{1}{c}{SR} & \multicolumn{1}{c}{RP} &      & \multicolumn{1}{c}{BA} & \multicolumn{1}{c}{PC} & \multicolumn{1}{c}{EC} & \multicolumn{1}{c}{SR} & \multicolumn{1}{c}{RP} &      & BA   & PC   & EC   & SR   & RP   &      & BA   & PC   & EC   & SR   & RP \bigstrut\\
		\hline
		3-4  &      & 40   & 40   & 40   & 40   & 40   &      & 40   & 40   & 38   & 40   & 40   &      & \multicolumn{1}{r}{1.4 } & \multicolumn{1}{r}{6.3 } & \multicolumn{1}{r}{231.0 } & \multicolumn{1}{r}{1.6 } & \multicolumn{1}{r}{0.3 } &      & \multicolumn{1}{r}{41 } & \multicolumn{1}{r}{385 } & \multicolumn{1}{r}{5923 } & \multicolumn{1}{r}{59 } & \multicolumn{1}{r}{0 } \bigstrut[t]\\
		4-4  &      & 40   & 40   & 40   & 40   & 40   &      & 39   & 36   & 15   & 39   & 40   &      & \multicolumn{1}{r}{328.9 } & \multicolumn{1}{r}{727.8 } & \multicolumn{1}{r}{2355.6 } & \multicolumn{1}{r}{267.4 } & \multicolumn{1}{r}{1.7 } &      & \multicolumn{1}{r}{26 } & \multicolumn{1}{r}{134 } & \multicolumn{1}{r}{7791 } & \multicolumn{1}{r}{56 } & \multicolumn{1}{r}{0 } \\
		5-4  &      & 31   & 23   & 24   & 28   & 40   &      & 14   & 10   & 2    & 12   & 40   &      & \multicolumn{1}{r}{2541.9 } & \multicolumn{1}{r}{2878.1 } & \multicolumn{1}{r}{3427.8 } & \multicolumn{1}{r}{2698.9 } & \multicolumn{1}{r}{29.2 } &      & \multicolumn{1}{r}{0 } & \multicolumn{1}{r}{150 } & \multicolumn{1}{r}{3647 } & \multicolumn{1}{r}{0 } & \multicolumn{1}{r}{0 } \\
		sum  &      & 111  & 103  & 104  & 108  & 120  &      & 93   & 86   & 55   & 91   & 120  &      & -    & -    & -    & -    & -    &      & -    & -    & -    & -    & - \bigstrut[b]\\
		\hline
	\end{tabular}%
\end{table*}%

\begin{table*}[!t]
	\centering
	\caption{{\color{black}Results of Customized Formulations With Height Limits}}
	\label{table: results_m3-IC-WH}
	\renewcommand\tabcolsep{4pt}
	\begin{tabular}{crrrrrrrrrrrrrcccccrccccc}
		\hline
		\multirow{2}[4]{*}{Case} &      & \multicolumn{5}{c}{\#Feasible}   &      & \multicolumn{5}{c}{\#Optimal}    &      & \multicolumn{5}{c}{Time(s)}      &      & \multicolumn{5}{c}{\#Nodes*} \bigstrut\\
		\cline{3-7}\cline{9-13}\cline{15-19}\cline{21-25}         &      & \multicolumn{1}{c}{BA} & \multicolumn{1}{c}{PC} & \multicolumn{1}{c}{EC} & \multicolumn{1}{c}{SR} & \multicolumn{1}{c}{RP} &      & \multicolumn{1}{c}{BA} & \multicolumn{1}{c}{PC} & \multicolumn{1}{c}{EC} & \multicolumn{1}{c}{SR} & \multicolumn{1}{c}{RP} &      & BA   & PC   & EC   & SR   & RP   &      & BA   & PC   & EC   & SR   & RP \bigstrut\\
		\hline
		3-4  &      & 40   & 40   & 40   & 40   & 40   &      & 40   & 40   & 40   & 40   & 40   &      & \multicolumn{1}{r}{3.5 } & \multicolumn{1}{r}{8.2 } & \multicolumn{1}{r}{140.1 } & \multicolumn{1}{r}{2.6 } & \multicolumn{1}{r}{0.4 } &      & \multicolumn{1}{r}{175 } & \multicolumn{1}{r}{348 } & \multicolumn{1}{r}{4745 } & \multicolumn{1}{r}{172 } & \multicolumn{1}{r}{1 } \bigstrut[t]\\
		4-4  &      & 40   & 40   & 39   & 40   & 40   &      & 37   & 33   & 22   & 37   & 40   &      & \multicolumn{1}{r}{515.7 } & \multicolumn{1}{r}{1115.1 } & \multicolumn{1}{r}{2139.5 } & \multicolumn{1}{r}{606.2 } & \multicolumn{1}{r}{3.3 } &      & \multicolumn{1}{r}{230 } & \multicolumn{1}{r}{561 } & \multicolumn{1}{r}{9227 } & \multicolumn{1}{r}{339 } & \multicolumn{1}{r}{1 } \\
		5-4  &      & 22   & 17   & 16   & 19   & 40   &      & 11   & 7    & 3    & 9    & 38   &      & \multicolumn{1}{r}{2697.3 } & \multicolumn{1}{r}{3100.5 } & \multicolumn{1}{r}{3353.8 } & \multicolumn{1}{r}{2844.4 } & \multicolumn{1}{r}{279.9 } &      & \multicolumn{1}{r}{0 } & \multicolumn{1}{r}{44 } & \multicolumn{1}{r}{1183 } & \multicolumn{1}{r}{0 } & \multicolumn{1}{r}{0 } \\
		sum  &      & 102  & 97   & 95   & 99   & 120  &      & 88   & 80   & 65   & 86   & 118  &      & -    & -    & -    & -    & -    &      & -    & -    & -    & -    & - \bigstrut[b]\\
		\hline
	\end{tabular}%
\end{table*}%

To focus on evaluating our MIP formulations under four different industrial considerations, we select 3 (out of 13) groups of instances to perform our computational studies, i.e., 120 instances in total from group 3-4, 4-4 and 5-4.
Results of four BRP-m3 MIP formulations (simply denoted by PC, EC, SR, RP, respectively), together with the basic MIP formulation (denoted by BA) on instances with corresponding modifications (with BA being tested on the basic instances) are shown in Tables \ref{table: results_m3-IC-NH} and \ref{table: results_m3-IC-WH}.
Note that column ``\#Nodes*'' presents the average number of B\&B nodes over instances solved to optimality by all the five formulations, which are actually the instances solved to optimality by EC.

Based on results in Tables \ref{table: results_m3-IC-NH} and \ref{table: results_m3-IC-WH}, it can be seen that our BRP-m3 is a flexible and effective basic model to build on.
For four extensions with practical considerations, except the one with energy cost, their performances are generally comparable to or better than that of the basic BRP-m3.
Although BRP-m3-PC has a different objective function and BRP-m3-SR and BRP-m3-RP have some new constraints, adding those complexities to BRP-m3 does not lead to a substantial degradation in its solution capability.
Actually, BRP-m3-RP, i.e., BRP-m3 with retrieval pace constraints, could perform orders-of-magnitude faster than that of the standard BRP-m3.
Hence, comparing to the specialized algorithms whose developments are generally challenging and demanding, BRP-m3, together with an MIP professional solver, is a user-friendly and effective platform to address more involved requirements arising from the practice with a rather stable performance.

Moreover, the superior performance demonstrated by BRP-m3-RP is worth a further investigation.
Comparing its optimal values to those of BRP-m3, we observe that they are very close.
As shown in Table \ref{table: results_m3-RP}, on 72\% of total 238 instances, BRP-m3-RP produces optimal values that are same as those of BRP-m3.
Even when they are different, at most 2 to 3 more relocations are involved.
Given BRP-m3-RP's drastically better computational speed, it would be beneficial to study how to use it to
exactly or approximately solve the basic BRP-m3 and its extensions.
Another observation is that the equalities reflecting the retrieval pace considerations, which are actually generalized upper bound (GUB) constraints \cite{Conforti-2014-Book}, play a critical role in generating strong cutting planes and in reducing the size of B\&B tree for a professional MIP solver.
Therefore, it inspires us to explore the BRP and derive similar constraints.
Specifically, we should derive bounds on the earliest and latest possible relocation turns in the moving sequence between which a block can be retrieved, and supply the related GUB constraints to BRP-m3 for fast computation.

\begin{table}[htbp]
	\centering
	\caption{{\color{black}Results of BRP-m3-RP Compared With BRP-m3}}
	\label{table: results_m3-RP}
	\renewcommand\tabcolsep{1.8pt}
	\begin{tabular}{crrrcrrrc}
		\hline
		\multirow{2}[4]{*}{Case} &      & \multicolumn{3}{c}{without a height limit} &      & \multicolumn{3}{c}{with height limits} \bigstrut\\
		\cline{3-5}\cline{7-9}         &      & \multicolumn{1}{c}{\#Optimal} & \multicolumn{1}{c}{\#RP=BA} & max(RP-BA) &      & \multicolumn{1}{c}{\#Optimal} & \multicolumn{1}{c}{\#RP=BA} & max(RP-BA) \bigstrut\\
		\hline
		3-4  &      & 40   & 36   & \multicolumn{1}{r}{2} &      & 40   & 36   & \multicolumn{1}{r}{2} \bigstrut[t]\\
		4-4  &      & 40   & 25   & \multicolumn{1}{r}{2} &      & 40   & 28   & \multicolumn{1}{r}{2} \\
		5-4  &      & 40   & 24   & \multicolumn{1}{r}{2} &      & 38   & 23   & \multicolumn{1}{r}{3} \\
		sum  &      & 120  & 85   & -    &      & 118  & 87   & - \bigstrut[b]\\
		\hline
	\end{tabular}%
\end{table}%

Finally, we would like to mention the new challenge from considering the energy cost in the BRP.
As shown in Section \ref{section_customizations},  in addition to simply counting the number relocations across all blocks, we must track every relocation movement associated with a particular block in the retrieval process.
Such a consideration renders BRP-m3-EC with a structure that is very much different from those of other BRP-m3 extensions and very hard to compute.
Indeed, note from Tables \ref{table: results_m3-IC-NH} and \ref{table: results_m3-IC-WH}, imposing a height limit is helpful to achieve a better computational performance in BRP-m3-EC, contrary to our previous understanding developed in all other experiments.
Given that the energy cost is a common concern among practitioners, it would be desired to carry out polyhedral and cutting plane studies to strengthen BRP-m3-EC for a better performance.
Alternatively, specialized algorithms, e.g., B\&B algorithms, could be developed to complement the current algorithm study on the BRPs.
}


\section{Conclusion}
\label{section_conclusion}

In this paper, we study the unrestricted BRP with distinct retrieval priorities, the complete retrieval and individual moves.
Our results include a general framework to derive strong lower bounds on the number of necessary relocations,
a set of demonstrations with respect to existing lower bounds, and a new but stronger one.
Moreover, we develop two exact computational methods: a new MIP formulation for the BRP, and a novel MIP formulation based iterative procedure.
{\color{black}
The MIP formulation is further customized into four extensions, each of which addresses a particular industrial consideration.
}
Our computational results show that the newly proposed lower bound greatly outperforms all existing ones in the literature, and is often less than the optimal value by just a couple of relocations.
Also, comparing to a recently published state-of-the-art formulation,  our two new computational methods demonstrate superior performances, especially on instances without a height limit, where our methods could be multi-order magnitude faster.
{\color{black}
Moreover, the customized MIP formulations display a stable performance in computing most of BRP instances, rather insensitive to additional complexities from different
industrial considerations.}

{\color{black}
Future research directions include identifying non-trivial subsets and their properties, and
continuing the tradition to derive stronger lower bounds under the presented general framework.
Also, as a new type of structural insights,
it is of a great interest to derive bounds on the earliest and latest possible relocation turns in the moving sequence between which a block can be retrieved.
Naturally, those bounds can be supplied to develop fast B\&B algorithms.
Regarding the new MIP formulations, one direction is to perform polyhedral studies to gain deep theoretical understandings and  to achieve computational improvements, especially for the formulation
considering blocks' movements and energy cost.
Another direction is to extend the presented formulations and the iterative procedure to solve other BRP variants with more practical considerations.
}


%

\appendices

{\color{black}
\section{Algorithms A5 and A5* for a Block Subset Satisfying Property 5}
\label{Appendix_Alg_P5}}

{\color{black}
To identify virtual layers satisfying P5, Algorithm A5 proceeds by first evaluating blocks in the top physical layer.
If they satisfy P5, we remove this layer from $\mathbb{C}$, and check the emerging top physical layer.
If a block in a physical layer causes it to violate P5, we replace the block with the one directly below it to construct an actual virtual layer.
We repeat the last step in the virtual layer with respect to P5, until either one virtual layer satisfying P5 is derived or no more virtual layer can be constructed.
Details of A5 are listed below. }

\noindent\rule[0.05\baselineskip]{3.5in}{0.5pt}
Algorithm A5: Identify a Block Subset Satisfying P5\\
\noindent\rule[0.45\baselineskip]{3.5in}{0.5pt}
\begin{algorithmic}[1]
	%
	\State {$\mathbb{B}^5 \leftarrow {\color{black}\text{the top layer of the current $\mathbb{C}$ }}$;
		$found \leftarrow $ false}
	\State{\textbf{wihle} $\mathbb{B}^5 \neq \emptyset$ and $found =$ false}
	\State {\hspace{0.5cm}$found \leftarrow$ true}
	\State {\hspace{0.5cm}\textbf{for} {\color{black} each block $i$ in $\mathbb{B}^5$}}
	\State {\hspace{1cm}\textbf{if} {\color{black}block $i$ causes $\mathbb{B}^5$ to violate P5}}
	\State {\hspace{1.5cm}\textbf{if} {\color{black}block $i$ is not on the floor}}
	\State {\hspace{2cm}update $\mathbb{B}^5$ by replacing $i$ with the block\\ \hspace{2cm}{underneath it}}
	\State {\hspace{1.5cm}\textbf{{\color{black}else}}}
	\State {\hspace{2cm}{\color{black}$\mathbb{B}^5 \leftarrow \emptyset$}}
	\State {\hspace{1.5cm}$found \leftarrow$ false; \textbf{break}}
	\State {\textbf{return} $\mathbb{B}^5$}
\end{algorithmic}
\noindent\rule[0.8\baselineskip]{3.5in}{0.5pt}

{\color{black}
Note that a WP block in a virtual layer definitely causes it to violate P5 if this block's  priority is higher than (or not lower than if duplicate priorities exist) the highest priority of blocks below the virtual layer.
Similarly, a BP block causes this layer to violate P5 if its priority is higher than (or not lower than if duplicate priorities exist) the lowest priority of all stacks after removing blocks above the virtual layer.}

The time complexity of A5 is $\mathcal{O}(BS)$. We reason it as follows.
{\color{black}
(\romannumeral1) Blocks are replaced by their underneath blocks at most $B$ times,
and at most $S$ blocks are checked during each replacement,
therefore blocks are evaluated at most $BS$ times
with a time complexity of $\mathcal{O}(1)$ each time.}
(\romannumeral2) If a block is replaced by its underneath block, we update the highest priority of blocks below the new virtual layer and the lowest priority of stacks after removing blocks above the new virtual layer.
(\romannumeral3) The two updates can be finished with $S - 1$ and $1$ comparison operations respectively,
if the highest priority of blocks below the underneath block is computed beforehand.
(\romannumeral4) The highest priority of blocks below each block can be computed as preprocessed data with a time complexity of $\mathcal{O}(B)$.
Following the calculation {\color{black}$BS + B (S - 1 + 1) + B = 2BS + B$},
we conclude the overall time complexity as $\mathcal{O}(BS)$.

{\color{black}
It can be easily seen that A5 tends to construct a virtual layer with blocks in the top physical layers. To consider other blocks in $\mathbb{C}$, we modify A5 to A5* that could lead to a different virtual layer and then a stronger lower bound. The basic idea is that, given a virtual layer constructed by the original A5,
we seek to update it with blocks piled at the lowest possible positions while ensuring its eligibility. The pseudo code of A5* is omitted here for its simplicity.
Its time complexity is the same as that of A5, i.e., $\mathcal{O}(BS)$.
}

\section{Algorithm A7 for a Block Subset Satisfying Property 7}
\label{Appendix_Alg_P7}

Given a WP block $i$ piled in a stack $s$,
we run Algorithm  A7 to identify a subset satisfying P7 in the following three steps.
(\romannumeral1) Pick $S - 1$ blocks in a way such that they are from distinct stacks in the other $S - 1$ stacks.
Together with block $i$, we form a virtual layer satisfying P5.
(\romannumeral2) Repeat step $(i)$ to form another virtual layer satisfying P5.
(\romannumeral3) Check whether a GB move of block $i$ is ensured.
The pseudo code is omitted here due to its similarity to that of Algorithm A5.
Also, the time complexity of A7 is the same as that of A5, i.e., $\mathcal{O}(BS)$.

{\color{black}
\section{Algorithm A8 for a Block Subset Satisfying Property 8}
\label{Appendix_Alg_P8}

Given a block $i$ piled in a stack $s$, we run Algorithm A8 to identify a subset $\mathbb{B}^8$ satisfying P8 in the following two steps.
Note that all the following operations are conducted on a copy of the initial configuration $\mathbb{C}$, which is denoted by $\mathbb{C}'$.

\textit{(1) Initialization.}
First pick the blocks above block $i$ that have priorities lower than that of $i$ to form subset $\mathbb{B}^{8,1}$.
Then pick a block with the highest priority from each of the other $S-1$ stacks, if not empty, to form subset $\mathbb{B}^{8,2}$.
Finally, set $\mathbb{B}^8 = \mathbb{B}^{8,1} \cup \mathbb{B}^{8,2}$.

\textit{(2) Verification.}
First remove blocks of $\mathbb{B} \backslash \mathbb{B}^8$ from $\mathbb{C}'$,
sort the $S-1$ stacks in a non-decreasing order of their priorities.
Then successively relocate blocks of $\mathbb{B}^{8,1}$ (in stack $s$) once, without considering the stack height limit, to convert them to be WP.
If a block can be relocated to be WP in multiple stacks, relocate it to a stack with the highest priority \cite{Tanaka-2018-COR}.
If it cannot be relocated to be WP, stop and conclude $\mathbb{B}^8$ satisfies P8.

Details of Algorithm A8 are shown as follows.

\noindent\rule[0.05\baselineskip]{3.5in}{0.5pt}
Algorithm A8: Identify a Block Subset Satisfying P8\\
\noindent\rule[0.45\baselineskip]{3.5in}{0.5pt}
\begin{algorithmic}[1]	
	\State {$\mathbb{B}^8 \leftarrow \mathbb{B}^{8,1} \leftarrow \mathbb{B}^{8,2} \leftarrow \emptyset;  \mathbb{C}' \leftarrow  \mathbb{C}$}
	\State {$\mathbb{B}^{8,1} \leftarrow$ blocks above block $i$ and with lower priorities}
	\State {$\mathbb{B}^{8,2} \leftarrow$ a block with the highest priority from each stack of set $\mathbb{S} \backslash \{s\}$}
	\State {remove blocks of $\mathbb{B} \backslash \{\mathbb{B}^{8,1} \cup \mathbb{B}^{8,2}\}$ from $\mathbb{C}'$}
	\State {sort stacks in the non-decreasing order of their priorities}%
	\State {\textbf{for} each block $j$ in $\mathbb{B}^{8,1}$}
	\State {\hspace{0.5cm}\textbf{if} block $j$ can be relocated to be WP}
	\State {\hspace{1cm}relocate block $j$ to the last feasible stack}
	\State {\hspace{0.5cm}\textbf{else}}
	\State {\hspace{1cm}$\mathbb{B}^8 \leftarrow \mathbb{B}^{8,1} \cup \mathbb{B}^{8,2}$;
		\textbf{break}}
	\State {\textbf{return} $\mathbb{B}^8$}
\end{algorithmic}
\noindent\rule[0.8\baselineskip]{3.5in}{0.5pt}

The time complexity of A8 is reasoned as follows.
(\romannumeral1) The initialization is with a time complexity of $\mathcal{O}(B)$.
(\romannumeral2) The verification is similar to checking the condition of Property 4,
therefore is with a time complexity of $\mathcal{O}(|\mathbb{B}^{8,1}| \log{S})$ if stacks are sorted beforehand \cite{Tanaka-2018-COR}.
(\romannumeral3) The time complexity of sorting stacks can be $\mathcal{O}(S \log{S})$.
Following the calculation $B + |\mathbb{B}^{8,1}| \log{S} + S \log{S}$,
we conclude that the time complexity of A8 is less than $\mathcal{O}(B\log{S})$.

Actually, the slight modification of LB-N considering the stack height limit \cite{Tanaka-2018-COR},
which is omitted in Section \ref{section_LB} to minimize distractions,
can be embedded into Algorithm A8.
Then we have the following proposition.

\begin{propositionA-1}
\label{prop_A8-generalizes-LB-N}
\label{prop_A8}
If a block subset $\mathbb{B}^4$ satisfying P4 is obtained in the derivation of LB-N,
a block subset $\mathbb{B}^8$ satisfying P8 can be derived by applying Algorithm A8 to some block.
\end{propositionA-1}

\begin{nproof}
Recall that LB-N is an iterative procedure that might remove blocks from the initial configuration.
Without loss of generality, for the current configuration, assume block $i$, which is piled in stack $s$, as the current target block when $\mathbb{B}^4$ is derived.
Then $\mathbb{B}^4 = \mathbb{B}^{4,1} \cup \mathbb{B}^{4,2}$,
where $\mathbb{B}^{4,1}$ includes the blocks above block $i$ and with lower priorities,
and $\mathbb{B}^{4,2}$ includes a block with the highest priority from each of the other $S-1$ stacks,
both in the current configuration.

In the initial configuration, we apply A8 to  block $i$ and set $\mathbb{B}^8 = \mathbb{B}^{8,1} \cup \mathbb{B}^{8,2}$,
where $\mathbb{B}^{8,1}$ includes the blocks above block $i$ and with lower priorities,
and $\mathbb{B}^{8,2}$ includes a block with the highest priority from each of the other $S-1$ stacks.

Since the initial configuration subsumes the current configuration,
it can be inferred that $\mathbb{B}^{8,1} \supseteq \mathbb{B}^{4,1}$ and the priority of the block in stack $s'$ and $\mathbb{B}^{8,2}$ is not lower than that of the block in stack $s'$ and $\mathbb{B}^{4,2}$ for $s'\in \mathbb{S}\backslash\{s\}$.
Clearly, if we cannot relocate all blocks of $\mathbb{B}^{4,1}$ to be WP, neither can we do those of $\mathbb{B}^{8,1}$.

Given that  $\mathbb{B}^4$ satisfies P4, i.e., some block(s) in $\mathbb{B}^{4,1}$ cannot be relocated to be WP, it follows that some block(s) in $\mathbb{B}^{8,1}$ also cannot be relocated to be WP, i.e., the condition of Property 8 is satisfied.
To conclude, a $\mathbb{B}^8$ satisfying P8 can be derived by applying A8 to block $i$.
\end{nproof}

\ \\
\indent We observe that a subset $\mathbb{B}^8$ satisfying P8 might contain some redundant block(s) that can be removed without losing P8.
To provide more flexibility to our lower bound derivation framework,  we reduce a $\mathbb{B}^8$ satisfying P8 to contain exactly $S$ stacks while ensuring this property.

We first define a key concept: \textit{barrier block}.
Considering the verification step in A8,
a block $j$ of $\mathbb{B}^{8,1}$ is piled in stack $s$, and is to be relocated to one of the other $S-1$ stacks.
If a higher prioritized block $k$ is already piled in one of those stacks,
then block $j$ cannot be relocated to that stack to be WP.
Hence we call block $k$ a \textit{barrier block} of block $j$.
For a block $j$ of $\mathbb{B}^{8,2}$,
which will not be relocated in the verification step  of A8,
we define its \textit{barrier blocks} as those in $\mathbb{B}^{8,2}$ that are with higher or equal  priorities.
Then we can have the following proposition.

\begin{propositionA-2}
Given a block subset satisfying P8, we can always pick exactly $S$ blocks out of it to build a new $\mathbb{B}^8$   satisfying P8 by the following procedure.

(\romannumeral1) Pick a block from stack $s$ that cannot be relocated to be WP in the verification step of A8.

(\romannumeral2) Pick a barrier block of the above picked block from the first of the $S-1$ sorted stacks.

(\romannumeral3) Pick a barrier block of the previously picked block from the second of the $S-1$ sorted stacks,
and repeat the process until a block is picked from the last of the $S-1$ sorted stacks.
\end{propositionA-2}

\begin{nproof}
We first show that the above procedure is feasible, i.e., the $S$ picked blocks are available.
Recall the relocation rule in the verification step of A8, a block is relocated to the highest possible prioritized stack, i.e., the latest possible stack.
Therefore, for a block picked from one of the $S-1$ sorted stacks, there must exists a barrier block of it in the next stack.

We then show that latter picked blocks are actually the barrier ones of former picked blocks.
It is obvious, since the barrier relationship is transitive.
For example, given the second picked block is a barrier one of first picked block,
and the third picked block is a barrier one of the second picked block,
we have the third picked block is also a barrier one of first picked block.
Therefore, when relocating a block, all blocks piled in the $S-1$ sorted stacks are its barrier blocks.

At last, we verify the reduced $\mathbb{B}^8$.
Recall the definition of barrier block, a block cannot be relocated to a stack to be WP if one of its barrier blocks is already in that stack.
Then, in the verification of the reduced $\mathbb{B}^8$, a block can only be relocated to an empty stack to be WP, since all blocks piled in the $S-1$ sorted stacks are its barrier blocks.
When relocating the last block, each of the $S-1$ stacks is occupied by a barrier block of it, hence it cannot be relocated to be WP.
In conclusion, the reduced $\mathbb{B}^8$ satisfies P8, and Proposition A-2 is proved.
\end{nproof}

Details of the reduction procedure, referred to as A8-s, are shown below.
As the reduction can be made easily and helps us to derive a non-BG relocation with a smaller $\mathbb{B}^8$,
we adopt it as a default step in A8 when performing the lower bound derivation.

\noindent\rule[0.05\baselineskip]{3.5in}{0.5pt}
Algorithm A8-s: Reduce a Block Subset Satisfying P8\\
\noindent\rule[0.45\baselineskip]{3.5in}{0.5pt}
\begin{algorithmic}[1]	
	\State {$j \leftarrow$ a block in stack $s$ that cannot be relocated to be WP}
	\State {$\mathbb{B}^8 \leftarrow \{j\}$}
	\State {\textbf{for} each stack $s'$ in $\mathbb{S} \backslash \{s\}$}
	\State {\hspace{0.5cm}$k \leftarrow$ a barrier block of block $j$ at the lowest tier}
	\State {\hspace{0.5cm}$j \leftarrow k$; $\mathbb{B}^8 \leftarrow \mathbb{B}^8 \cup \{j\}$}
\end{algorithmic}
\noindent\rule[0.8\baselineskip]{3.5in}{0.5pt}

\section{Proof of Proposition 1}
\label{Appendix_Proof-Proposition1}
As for relaxing integer variables to continuous ones, it can be easily verified that  there exists one optimal solution such that those continuous variables take integer values, due to the integrality of related variables. So, we focus on detailed proofs of the other two statements.

\begin{nproof}
(b) As for eliminated constraints in (\ref{cons_y_up_3}), (\ref{cons_z_1}) and (\ref{cons_z_2}), they are actually implied or dominated by (\ref{cons_y_dn_2}), (\ref{cons_y_dn_3}), (\ref{cons_x_2}), (\ref{cons_x_3}) and (\ref{cons_y_dn_4}) as shown below.

First, if $\hat{y}_{ij}^t = 0$,  (\ref{cons_y_up_3}) trivially holds.
Otherwise, i.e., $\hat{y}_{ij}^t = 1$, we can derive $\check{y}_{ij}^t = 0$ from (\ref{cons_y_dn_2}) and (\ref{cons_y_dn_3}), and can further derive (\ref{cons_y_up_3}) from (\ref{cons_x_2}) and (\ref{cons_x_3}) as follows.
\begin{align}
& \hat{y}_{ij}^t = x_{ij}^{t\text{-}1} - x_{ij}^t \leq  x_{ij}^{t\text{-}1}
\hspace{1.2cm}
\forall i \in \mathbb{B};\
j \in \mathbb{B}',\ j < i;\
t \in \mathbb{T}
\notag\\
& \hat{y}_{ij}^t = x_{ij}^{t\text{-}1} - x_{ij}^t - z_{ij}^t \leq  x_{ij}^{t\text{-}1}
\hspace{0.35cm}
\forall i \in \mathbb{B};\
j \in \mathbb{B}',\ j > i;\
t \in \mathbb{T}
\notag
\end{align}
Hence, in either case,  (\ref{cons_y_up_3}) can be derived from (\ref{cons_y_dn_2}), (\ref{cons_y_dn_3}), (\ref{cons_x_2}) and (\ref{cons_x_3}).

Second, (\ref{cons_z_1}) can be derived from (\ref{cons_x_3}) as follows.
\begin{align}
& z_{ij}^t \mathrm{=} x_{ji}^{t\text{-}1} \text{--} \hat{y}_{ji}^t \mathrm{+} \check{y}_{ji}^t \text{--} x_{ji}^t
\leq x_{ji}^{t\text{-}1} \text{--} \hat{y}_{ji}^t \mathrm{+} \check{y}_{ji}^t
\
\forall i \mathrm{\in} \mathbb{B};
j \mathrm{\in} \mathbb{B}', j \mathrm{>} i;
t \mathrm{\in} \mathbb{T}
\notag
\end{align}

Third, we use the contradiction to show that (\ref{cons_z_2})  can be eliminated. Assuming that a constraint in (\ref{cons_z_2}) is violated,  there exists a block $i$, which is piled beneath a lower prioritized block $k$ after the $t^{th}$ relocation move (hence $x_{ki}^{t\text{-}1} - \hat{y}_{ki}^t + \check{y}_{ki}^t = 1$ and $\sum_{j \in \mathbb{B}' \backslash \{i\}} {(x_{ij}^{t\text{-}1} {\rm{-}} \hat{y}_{ij}^t {\rm{+}} \check{y}_{ij}^t)} = 1$),
that is retrieved during turn $t$ (hence $\sum_{j \in \mathbb{B}', j > i} {z_{ij}^t} = 1$).
Then we can infer the following two constraints from (\ref{cons_x_2}) and (\ref{cons_x_3}).
\begin{align}
\sum_{j \in \mathbb{B}' \backslash \{i\}} {x_{ij}^t} & {\rm{=}}
\sum_{j \in \mathbb{B}', j < i}{x_{ij}^t} {\rm{+}}
\sum_{j \in \mathbb{B}', j > i} {x_{ij}^t}
\notag\\
& {\rm{=}} \sum_{j \in \mathbb{B}' \backslash \{i\}} {(x_{ij}^{t\text{-}1} {\rm{-}} \hat{y}_{ij}^t {\rm{+}} \check{y}_{ij}^t)} {\rm{-}}
\sum_{j \in \mathbb{B}', j > i} {z_{ij}^t} {\rm{=}}
1 {\rm{-}} 1 {\rm{=}} 0
\notag\\
\sum_{j \in \mathbb{B} \backslash \{i\}} {x_{ji}^t} & {\rm{=}}
\sum_{j \in \mathbb{B} \backslash \{i, k\}} {x_{ji}^t} + x_{ki}^t
\ \ \ (\text{note that $k > i$})
\notag\\
& {\rm{=}} \sum_{j \in \mathbb{B} \backslash \{i, k\}} {x_{ji}^t} + (x_{ki}^{t\text{-}1} {\rm{-}} \hat{y}_{ki}^t {\rm{+}} \check{y}_{ki}^t)\geq
0 {\rm{+}} 1 {\rm{=}} 1
\notag
\end{align}
Considering (\ref{cons_y_dn_4}), we have the following constraint.
\begin{align}
\sum_{j \in \mathbb{B} \backslash \{i\}} {\check{y}_{ji}^{t\text{+}1}} {\rm{\leq}}
\sum_{j \in \mathbb{B}' \backslash \{i\}} {x_{ij}^t} {\rm{-}} \sum_{j \in \mathbb{B} \backslash \{i\}} {x_{ji}^t} {\rm{\leq}}
0 {\rm{-}} 1,
\notag
\end{align}
which causes the formulation infeasible.
Hence solutions satisfying (\ref{cons_x_2}), (\ref{cons_x_3}) and (\ref{cons_y_dn_4}) naturally satisfy (\ref{cons_z_2}), which indicates the latter one is not necessary.
\\

(c) As for the set of new constraints.
Note that it does not hurt to require that block 1 is retrieved as soon as possible in the retrieval process.
Such consideration is actually reflected in constraints (\ref{cons_e_1})-(\ref{cons_e_7}).
Specifically,  equalities in (\ref{cons_e_1}) ensure that block 1 is retrieved whenever the one directly piled upon it is relocated.
Obviously, that block is relocated from the top of block 1 exactly once, i.e., (\ref{cons_e_2}).
Since block 1 is retrieved, no other block can be relocated to the top of it, i.e., (\ref{cons_e_3}) and (\ref{cons_e_4}), and it cannot be relocated or retrieved, i.e., (\ref{cons_e_5}) and (\ref{cons_e_6}).

Let $\tau$ denote the turn in which block 1 is retrieved.
Clearly, we have block 1 remained in its initial position for $t\leq \tau$, which gives $u_1^t=h_1$ for $t\leq \tau$.
For $t > \tau$, since $x_{1j}^{t\text{-}1} = \hat{y}_{1j}^t =\check{y}_{1j}^t=0$ for all $j \rm{\neq} 1$,
constraints in (\ref{cons_u_1}), i.e., the only constraints restricting $u_i^t$, reduce to $u_1^t \geq u_j^t + 1 - H$ for all $j \neq 1$.
It is trivially true given the fact that its right-hand-side is less than or equal to $1$ and $u_1^t\geq 1$.
Hence, imposing (\ref{cons_e_7}) for $t\in \mathbb{T}$ will not eliminate any optimal move sequence.
\end{nproof}
}

\section{Enhanced Iterative Scheme}
\label{Appendix_Enhanced-Iterative-Scheme}

{\color{black}
Since the fast computation heuristics and  the reparation heuristics are rather simple, we present their main ideas and omit the detailed pseudo codes.
Following the myopic strategy, the two fast heuristics  determine the next move by minimizing the number of direct blockages or the value of LB$_4$ of the resulting configuration for each turn.
On the other hand, the reparation heuristic seeks to fix a solution violating the height limit by repairing some relocation moves.
A conservative strategy is used first:
if a relocation move leads to a block over the height limit,
we change its destination to a stack of a lower height,
while ensuring that all the following moves of the whole block set remain applicable.
If the strategy fails at a turn, switch to a more aggressive strategy: to
generate new moves for the failed turn and all the following  turns by one of the fast heuristics.
}

\noindent\rule[0.05\baselineskip]{3.5in}{0.5pt}
Algorithm IS*: Enhanced Iterative Scheme\\
\noindent\rule[0.45\baselineskip]{3.5in}{0.5pt}

\begin{algorithmic}[1]
\State {$L \leftarrow 0$, $L' \leftarrow \text{a lower bound of BRP-m3}$}
\State {\textbf{while} $L < L'$}
\State {\hspace{0.5cm}$L \leftarrow L'$}
\State {\hspace{0.5cm}generate initial solutions by two fast heuristics}
{\color{black}
\State {\hspace{0.5cm}\textbf{if} {objective value of the best initial solution $sln1$ is $L$}}
\State {\hspace{1cm}\textbf{break}}
}
\State {\hspace{0.5cm}update and compute BRP-m3R without a height limit}
\State {\hspace{0.5cm}$(sln1,\ L') \leftarrow \text{the optimal solution and objective value}$}
\State {\textbf{if} $sln1$ satisfies the height limit}
\State {\hspace{0.5cm}\Return $sln1$}
\State {repair $sln1$ and get $sln2$}
\State {\textbf{if} the objective value of $sln2$ equals $L$}
\State {\hspace{0.5cm}\Return $sln2$}
\State {$L \leftarrow 0$}
\State {\textbf{while} $L < L'$}
\State {\hspace{0.5cm}$L \leftarrow L'$}
\State {\hspace{0.5cm}generate initial solutions by two fast heuristics}
{\color{black}
\State {\hspace{0.5cm}\textbf{if} objective value of the best initial solution $sln3$ is $L$}
\State {\hspace{1cm}\textbf{break}}
}
\State {\hspace{0.5cm}update and compute BRP-m3R with a height limit}
\State {\hspace{0.5cm}$(sln3,\ L') \leftarrow \text{the optimal solution and objective value}$}
\State {\Return {$sln3$}}

\end{algorithmic}
\noindent\rule[0.8\baselineskip]{3.5in}{0.5pt}

\section*{Acknowledgment}
{\color{black}
The authors thank the associate editor and the reviewers for their constructive comments that prompt us to discuss the applicability and flexibility of our methods, and improve the clarity of some details.
}
The authors thank Shunji Tanaka for providing instances, codes and games on his website (\url{https://sites.google.com/site/shunjitanaka/brp}) that are very helpful to understand the BRP.
The authors also thank Ting Li for her constructive suggestions on this research topic.

\ifCLASSOPTIONcaptionsoff
\newpage
\fi

\end{document}